\documentclass[11pt]{amsart}
\usepackage[hypertex]{hyperref}
\usepackage{cite}
\usepackage{amssymb}
\usepackage{amsmath}
\usepackage{amsthm}
\usepackage{amsfonts}
\usepackage{graphicx}

\usepackage{color}

\usepackage[toc,page]{appendix}

\newcommand{\beq}{\begin{equation}}
\newcommand{\eeq}{\end{equation}}
\newcommand{\bea}{\begin{eqnarray}}
\newcommand{\eea}{\end{eqnarray}}
\newcommand{\beas}{\begin{eqnarray*}}
\newcommand{\eeas}{\end{eqnarray*}}

\newcommand{\G}{G}

\newtheorem{thm}{Theorem}[section]
\newtheorem{cor}[thm]{Corollary}

\newtheorem{lem}[thm]{Lemma}
\newtheorem{nota}[thm]{Notation}
\newtheorem{prop}[thm]{Proposition}

\theoremstyle{definition}

\newtheorem{df}[thm]{Definition}

\newtheorem{example}{Example}[section]
\newtheorem{rem}[thm]{Remark}

\numberwithin{equation}{section}

\title[Quasi-invariance on infinite-dimensional Heisenberg
groups]{Quasi-invariance for heat kernel measures on sub-Riemannian
infinite-dimensional Heisenberg groups}
\author[Baudoin]{Fabrice Baudoin{$^{\dagger }$}}
\thanks{\footnotemark {$^\dagger$} This research was supported in part by NSF
Grant DMS-0907326.}
\address{Department of Mathematics \\
Purdue University \\
West Lafayette, IN 47907 USA} \email{fbaudoin@math.purdue.edu}
\author[Gordina]{Maria Gordina{$^{*}$}}
\thanks{\footnotemark {$^*$} This research was supported in part by NSF Grant DMS-1007496.}
\address{Department of Mathematics\\
University of Connecticut\\
Storrs, CT 06269, USA} \email{maria.gordina@uconn.edu}
\author[Melcher]{Tai Melcher{$^{\dagger\dagger}$}}
\thanks{\footnotemark {$^{\dagger\dagger}$} This research was supported in part by NSF
Grant DMS-0907293}
\address{Department of Mathematics\\
University of Virginia\\ Charlottesville, VA 22903 USA}
\email{melcher@virginia.edu}


\newcommand{\p}{\partial}

\begin{document}

\begin{abstract}

We study heat kernel measures on sub-Riemannian infinite-dimensional
Heisenberg-like Lie groups.  In particular, we show that Cameron-Martin
type quasi-invariance results hold in this subelliptic setting and
give $L^p$-estimates for the Radon-Nikodym derivatives. The main
ingredient in our proof is a generalized
curvature-dimension estimate which holds on approximating finite-dimensional
projection groups. Such estimates were first introduced by Baudoin
and Garofalo in \cite{BaudoinGarofalo2011}.
\end{abstract}

\maketitle
\tableofcontents

\section{Introduction}

We prove Cameron-Martin type quasi-invariance results for subelliptic
heat kernel measures on infinite-dimensional Heisenberg-like groups.  These
groups were first defined in \cite{DriverGordina2008} and quasi-invariance
was proved for elliptic heat kernel measures in this setting.
Quasi-invariance results are of interest, for example, in the study
of smoothness of measures on infinite-dimensional spaces.  In finite
dimensions, one typically defines smoothness as absolute
continuity with respect to some reference measure and smoothness
of the associated density.   In infinite dimensions, in the absence
of a canonical reference measure, alternative interpretations of
smoothness must be made; see for example
\cite{MR2152244, Driver2003,Malliavin1990a,MattinglyPardoux2006}.  In particular, in an
infinite-dimensional setting, it is natural to interpret quasi-invariance
as a smoothness property.

Quasi-invariance of heat kernel measures in infinite dimensions has
previously been the subject of much study in elliptic settings; see
for example the review \cite{Driver1995a} and references therein.  Typically
the proofs in the elliptic case rely on lower bounds on the Ricci
curvature (as was the case in \cite{DriverGordina2008}); of course
such lower bounds are unavailable in a subelliptic setting.  To the
authors' knowledge, this is the first quasi-invariance result in
an infinite-dimensional subelliptic setting.

\subsection{Statement of Results}

Let $(W,H,\mu)$ be an abstract Wiener space and let $\mathbf{C}$
be a finite-dimensional inner product space.  Let
$\mathfrak{g}=W\times \mathbf{C}$ be an infinite-dimensional
Heisenberg-like Lie algebra, which is constructed as an infinite-dimensional
step 2 nilpotent Lie algebra with continuous Lie bracket satisfying
the following condition:
\begin{equation}
\label{e.h}
[W,W]=\mathbf{C}.
\end{equation}
Let $G$ denote $W\times\mathbf{C}$ thought of as a group with
operation \[ g_1\cdot g_2 = g_1 + g_2 + \frac{1}{2}[g_1,g_2]. \]
Then $G$ is a Lie group with Lie algebra $\mathfrak{g}$, and $G$
contains the subgroup $G_{CM} = H\times\mathbf{C}$ which has Lie
algebra $\mathfrak{g}_{CM}$.  See Section \ref{s.3} for
definitions and details.

Now let $\{B_t\}_{t\ge0}$ be a Brownian motion on $W$.  The
solution to the stochastic differential equation
\begin{equation}
\label{e.1}
dg_t = g_t\circ dB_t\qquad \text{ with } g_0=e
\end{equation}
is a Brownian motion on $G$, which is defined explicitly in Proposition
\ref{p.Mn} and Definition \ref{d.bm}.  For all $t>0$, let
$\nu_t=\mathrm{Law}(g_t)$ denote the heat kernel measure at time
$t$.  At this point, let us briefly comment that, as mentioned in \cite{GordinaMelcher2011},
it may seem at first glance that the restriction
$\mathrm{dim}(\mathbf{C})<\infty$ implies that this subelliptic
example is in some sense only finitely many steps from being elliptic.
However, this is truly a subelliptic model and
the topologies one must deal with in this setting
significantly change the standard analysis and introduce
several non-trivial complications not present in the elliptic
case.  For further discussion, see Section 1.3 of \cite{GordinaMelcher2011}.

The main results of the present paper are presented in Section \ref{s.qi}.
Namely, in Theorem \ref{t.qi}, we prove that $\nu_t$ is quasi-invariant under
translation by elements of $G_{CM}$ and obtain bounds on the $L^p$-norms of
the Radon-Nikodym derivatives.  Then given the equivalence of measures and the
$L^p$-estimates on the associated densities, we are able to immediately prove
in Corollary \ref{c.5.19} that the semi-group is strong Feller on
$G_{CM}$.

To put these results in context, again recall that
if $W$ is finite-dimensional, then (\ref{e.h}) implies that
$\mathrm{span}\{(\xi_i,0),[(\xi_i,0),(\xi_j,0)]\}=\mathfrak{g}$,
where $\{\xi_i\}_{i=1}^{\mathrm{dim}(W)}$ is some orthonormal basis
of $W$, and thus we have satisfaction of H\"ormander's
condition.  This then implies that $\nu_t$ is a smooth measure, in
the sense that $\nu_t$ is absolutely continuous with respect to
Haar measure on $G=W\times\mathbf{C}$ and its density is a smooth
function on $G$.  If $W$ and thus $G$ is infinite-dimensional, however, the
notion of smoothness is not so well-defined, and
quasi-invariance may be interpreted
as a first step toward proving $\nu_t$ is a ``strictly positive''
smooth measure.

\subsection{Discussion of proofs}
Functional inequalities provide a powerful tool to study the problem
of the equivalence of  heat kernel measures. In particular, it is
a well-known fact that, on a finite-dimensional
complete Riemannian manifold $\mathbb{M}$ with non-negative Ricci
curvature, the heat semi-group $\{P_t\}_{t \ge 0}$ satisfies the Harnack
type inequality
\begin{align}
\label{W-Harnack}
(P_t f)^\alpha (x) \le P_t f^\alpha (y) \exp \left( \frac{\alpha}{\alpha -1} \frac{d^2(x,y)}{4t} \right),
\end{align}
where $x,y \in \mathbb{M}$, $f \in L^\infty(\mathbb{M})$ with  $f \ge 0$, and
$\alpha >1$  (see for example \cite{Wang1997a}).  Using the above inequality with indicator functions
immediately implies that the heat kernels measures   $p_t (x, dz)$
and $p_t(y,dz)$ are equivalent. Of course, in a finite-dimensional
framework,  the latter is obvious and may be seen  from the
positivity of the heat kernel. But the relevant fact here is that
the functional inequality \eqref{W-Harnack} is independent of the
dimension of the manifold $\mathbb{M}$ and we may therefore hope
that it holds even in some infinite-dimensional settings, where equivalence of
measures is a highly non-trivial problem.

Again, a lower bound on the Ricci curvature of $\mathbb{M}$
typically plays a major role in the proof of inequalities like
\eqref{W-Harnack} and such bounds are unavailable in our subelliptic
setting.  However, in a recent work \cite{BaudoinGarofalo2011}, Baudoin and
Garofalo introduced a generalized curvature-dimension inequality
that holds in a general class of sub-Riemannian settings. The main idea is
to control sub-Riemannian curvature quantities both in the horizontal
and the vertical directions. In the present paper, we prove that a uniform generalized
curvature-dimension inequality holds on appropriate finite-dimensional
approximation groups of $G$, and that, as a consequence, a uniform
version of \eqref{W-Harnack} holds on the finite-dimensional
approximation groups.  A by-product of this is a Cameron-Martin
type quasi-invariance result for the heat kernel measure on $G$.

Let us emphasize that our approach is actually quite general
and does not rely on the specific nature of the infinite-dimensional
Heisenberg-like groups we consider here. The main ingredient is the
existence of \textit{good} finite-dimensional approximations on
which uniform generalized curvature-dimension bounds hold.
As was done in \cite{DriverGordina2009} in the elliptic setting, one could
significantly generalize the method to include other infinite-dimensional
subelliptic settings.

The organization of the paper is briefly as follows.
In Section \ref{s.finite} we review results for subelliptic heat kernels on finite-dimensional
Lie groups under the assumption of generalized curvature-dimension estimates
and satisfaction of certain commutation relations.  In particular, in Section \ref{s.finrev}, we show
that under these assumptions reverse Poincar\'e and log Sobolev
estimates hold, and in Section \ref{s.finhar} we show how these
estimates in turn give Wang type and integrated Harnack inequalities.
The ideas in this section in a finite-dimensional setting are
discussed in more detail and greater generality in \cite{BB11}.

In Section \ref{s.infheis} we review the definition of the
infinite-dimensional Heisenberg-like groups first considered in
\cite{DriverGordina2008}.  In the present paper, we choose a
topological structure that  is better adapted to subellipticity,
and in this way our construction parallels \cite{GordinaMelcher2011},
where subelliptic heat kernel measure was also studied.  In this
section, we also review the properties of infinite-dimensional
Heisenberg-like groups required for the sequel, as well as recalling
the Cameron-Martin subgroup and the finite-dimensional projection
groups which serve as approximations.

In Section \ref{s.Pcurv} we show that the curvature-dimension estimates and
commutation relations considered in Section \ref{s.finite} are
satisfied on the infinite-dimensional Heisenberg-like groups and their
finite-dimensional projections, and thus in Section \ref{s.Pin} we
are able to show that reverse Poincar\'e and log Sobolev inequalities,
as well as Wang type Harnack inequalities, hold on the finite-dimensional
projection groups.

Finally, in Section \ref{s.BM} we review the construction of
subelliptic heat kernel measures on the infinite-dimensional Heisenberg-like
groups, and in Section \ref{s.qi} we show how the Wang type Harnack
inequalities proved in Section \ref{s.Pin} for the finite-dimensional
projection groups imply quasi-invariance of the subelliptic heat
kernel measure under translation by elements of the Cameron-Martin space,
as well as yielding $L^p$-estimates for the Radon-Nikodym derivatives.  We
also show how the quasi-invariance and $L^p$-estimates then immediately imply
that $P_t$ is strong Feller on $G_{CM}$.

{\bf Acknowledgements. } The second author wishes to thank M.
R\"{o}ckner who asked about the strong Feller property in this
context and thus inspired the inclusion of Corollary \ref{c.5.19}.

\section{Functional inequalities on finite-dimensional Lie groups}
\label{s.finite}

For the whole of this section, $G$ will denote a real {\em
finite-dimensional} connected unimodular Lie group with Lie algebra
$\mathfrak{g}$ and identity element $e$.  Let $dx$ denote bi-invariant
Haar measure on $G$.  For $x\in G$, let $L_x$ and $R_x$ denote left and right
translation by $x$, respectively.  For $A\in\mathfrak{g}$,
let $\tilde{A}$ denote the unique left invariant vector field on
$G$ such that $\tilde{A}(e)=A\in\mathfrak{g}$, that is,
$\tilde{A}(g)=L_{x*}A$.

We will assume there is a linearly independent collection
$\{X_i\}_{i=1}^n\subset\mathfrak{g}$ for which there exists some $r$ such that
\begin{multline}
\label{e.hc}
\mathfrak{g} = \mathrm{span}\{X_{i_1},[X_{i_1},X_{i_2}],\ldots,
		[X_{i_1},[X_{i_2},\cdots,[X_{i_{r-1}},X_{i_r}]\cdots]]: \\
	i_1,\ldots,i_r=1,\ldots,n\}.
\end{multline}
We will refer to
$\mathcal{H}:=\mathrm{span}(\{X_i\}_{i=1}^n)$ as the {\em horizontal
directions}, and we will suppose that $\mathfrak{g}$ is equipped with an
inner product for which $\{X_i\}_{i=1}^n$ is an orthonormal
basis of $\mathcal{H}$.
In particular, (\ref{e.hc}) implies that $\{\tilde{X}_i\}_{i=1}^n$ is a H\"ormander
set of vector fields, in the sense that the $\tilde{X}_i$'s and all their
commutators up to order $r$ generate $T_gG$ for all $g\in G$.  Thus, by the
classical H\"ormander's theorem \cite{Hormander67} the
second order differential operator
\[ L = \sum_{i=1}^n \tilde{X}_i^2 \]
is hypoelliptic, and, for all $t>0$, there exists a smooth kernel
$p_t:G\times G\rightarrow\mathbb{R}^+$ so that
\[ P_tf(x):=e^{t\bar{L}/2}f(x) = \int_G f(y)p_t(x,y)\,dy, \]
for all $f\in L^2(G, dy)$, where $\bar{L}$ denotes the $L^2(G,dy)$
closure of $L|_{C_c^\infty(G)}$.  The heat semi-group $\{P_t\}_{t>0}$ is a
symmetric Markov semi-group, and, since $p_t$ is smooth, the mapping $(t,x)\mapsto
P_tf(x)$ is smooth on $(0,\infty)\times G$ for any $f\in L^p(dx)$,
$p\in[1,\infty]$.

In a slight abuse of notation,
we will write that $p_t(e,x)=p_t(x)$. In particular, the left
invariance of $L$ implies that $p_t$ is a left convolution kernel
and thus $p_t(x,y)=p_t(xy^{-1})$.  We call the measure $p_t\,dy$
the {\em heat kernel measure} on $G$ associated to $L$.
By construction, the operator $P_t$ commutes with left translations, and since
our Lie group is unimodular the Haar measure is bi-invariant.  Thus,
$p_t(x,y) = p_t(e,x^{-1}y)$.
Since it is additionally known that $p_t$ is a symmetric kernel, we have the
following lemma of some standard properties of the heat kernel.
\begin{lem}
\label{l.pt}
For all $t>0$ and $x,y\in G$,
\begin{enumerate}
\item $p_t(x,y)=p_t(x^{-1}y)=p_t(y^{-1}x)$, and
\item $p_t(x^{-1})=p_t(x)$.
\end{enumerate}
\end{lem}

Another well-known interpretation of the heat kernel measure $p_t\,dy$
is as the distribution in $t$ of Brownian motion on $G$.
For $\{b_i\}_{i=1}^n$ independent real-valued
Brownian motions, $b_t=\sum_{i=1}^n b_t^iX_i$ is a Brownian motion
on $\mathcal{H}$.  Now for $t>0$ let $g_t$ denote the solution
to the following Stratonovitch stochastic differential equation:
\[ dg_t = g_t\circ db_t := L_{g_t*}\circ db_t = \sum_{i=1}^n \tilde{X}(g_t)\circ db_t^i,
	\qquad \text{ with } g_0 = e. \]
Then, $\{g_t\}_{t\ge0}$ is a Brownian motion on $G$ started at the identity,
and, for all $t>0$, $\mathrm{Law}(g_t)=p_t\,dy$.

For $f,g\in C^\infty(G)$, we define the standard differential forms
\begin{align*}
\Gamma\left( f, g \right)
	&:= \frac{1}{2} (L(fg) -fLG -gLf)
	= \sum_{j=1}^{n} \left(\tilde{X}_i f \right)
		\left(\tilde{X}_i g\right) \text{ and} \\
\Gamma_{2}\left( f, g \right)
	&:= \frac{1}{2}\left( L\Gamma\left( f, g \right)-\Gamma\left( f, Lg\right)-\Gamma\left( g,
		Lf\right)\right).
\end{align*}
We follow the usual notational convention that $\Gamma\left( f \right):=\Gamma\left( f, f
\right)$ and $\Gamma_2(f):= \Gamma_2(f,f)$.
We will also need to consider derivatives in non-horizontal, or {\it
vertical}, directions.
Let $\mathcal{V}\subset\mathfrak{g}$ denote the Lie subalgebra such that
\[ \mathfrak{g} = \mathcal{H}\oplus\mathcal{V}, \]
and let $\{Z_\ell\}_{\ell=1}^d$ denote an orthonormal basis of
$\mathcal{V}$, where $d=\mathrm{dim}(\mathcal{V})$,
and following \cite{BakBau} and \cite{BaudoinGarofalo2011} we define
\begin{align*}
\Gamma^{Z}\left( f, g \right)
	&:= \sum_{\ell=1}^{d} \left(\tilde{Z}_\ell f \right) \left(\tilde{Z}_\ell g \right)
		\text{ and} \\
\Gamma_2^Z(f,g) &:= \frac{1}{2} \left(L\Gamma^Z(f,g) - \Gamma^Z(f,Lg) -
	\Gamma^Z(g,Lf)\right).
\end{align*}
Again we will let $\Gamma^Z(f):=\Gamma^Z(f,f)$ and $\Gamma^Z_2(f):=\Gamma^Z_2(f,f)$.

\subsection{Reverse inequalities under curvature bound assumptions}
\label{s.finrev}

In this section we consider some implications of certain curvature
bound assumptions on $G$.  These curvature bounds first appear
in \cite{BaudoinGarofalo2011} and serve as generalizations of the standard
curvature-dimension inequalities which appear in the elliptic literature.
The assumption may be stated as follows: Suppose that there exist
$\alpha,\beta>0$ such that, for any $\nu>0$ and $f\in C^\infty(G)$,
\begin{equation}
\label{e.curv}
\Gamma_2(f) + \nu\Gamma_2^Z(f)
	\ge \alpha\Gamma^Z(f) - \frac{\beta}{\nu}\Gamma(f).
\end{equation}
Before proceeding with the primary results of this section, we
prove the following lemma which will be helpful in the sequel.

\begin{lem}
\label{l.simple}
Suppose $\phi:J\rightarrow\mathbb{R}$ is
a smooth function on an open interval $J\subset\mathbb{R}$ and
$f:G\rightarrow\mathbb{R}$ is a measurable function.  Fix $T>0$ and set
\[ \Sigma(t,x) = P_t(\phi(P_{T-t}f))(x), \]
for all $t\in[0,T]$ and $x\in G$ (assuming $P_{T-t}f(G)\subset J$).  Then
\[ \frac{d\Sigma}{dt} = \frac{1}{2}P_t(\phi''(P_{T-t}f)\Gamma(P_{T-t}f))(x). \]
\end{lem}

\begin{proof}
For simplicity, set $u_t=P_{T-t}f$.  Then we just compute
\begin{align}
\frac{d\Sigma}{dt}
	= P_t\left(L\phi(u_t)/2 + \frac{d}{dt}\phi(u_t)\right)
	\label{e.simple}
	= \frac{1}{2}P_t\left(L\phi(u_t) - \phi'(u_t)Lu_t\right).
\end{align}
Note that
\begin{align*}
L\phi(u_t) &= \sum_{i=1}^n \tilde{X}_i^2 \phi(u_t)
	= \sum_{i=1}^n \tilde{X}_i \left(\phi'(u_t)(\tilde{X}_iu_t)\right) \\
	&= \sum_{i=1}^n \left( \phi''(u_t)(\tilde{X}_iu_t)^2 +
		\phi'(u_t)(\tilde{X}_i^2u_t)\right)
	= \phi''(u_t)\Gamma(u_t) + \phi'(u_t)Lu_t.
\end{align*}
Then combining this with \eqref{e.simple} yields the desired result.
\end{proof}

We now prove that, assuming the curvature bound stated above, a reverse Poincar\'e
inequality holds on $G$.

\begin{nota}
\label{n.C}
Let $\mathcal{C}$ denote the set of functions $f:G\rightarrow\mathbb{R}$
such that $f\in C^\infty(G)\cap L^\infty(G)$ and
$f,\sqrt{\Gamma(f)},\sqrt{\Gamma^Z(f)}\in L^2(G)$.  Note for example that
$\mathcal{C}$ includes all smooth functions with compact support.  We will also
let $\mathcal{C}^+$ denote functions $f$ such that $f=g+\varepsilon$ for some
$g\in \mathcal{C}$ with $g\ge0$ and $\varepsilon>0$.
\end{nota}

\begin{rem}
It is shown in \cite{BaudoinGarofalo2011} that these function spaces
are stable under $P_t$; that is, if $f\in\mathcal{C}$ then
$P_tf\in\mathcal{C}$ for all $t>0$, and similarly for $\mathcal{C}^+$.
\end{rem}

\begin{prop}
\label{p.rpoin}
Assume that a curvature bound is satisfied as stated above in
\eqref{e.curv}.  Then, for all $T>0$ and $f\in \mathcal{C}$,
\[ \Gamma(P_Tf) + \alpha T\Gamma^Z(P_Tf)
	\le \frac{1 + \frac{2\beta}{\alpha}}{T}(P_T(f^2)-(P_Tf)^2). \]
In particular, this implies that
\[ \Gamma(P_Tf) \le \frac{1 + \frac{2\beta}{\alpha}}{T}(P_T(f^2)-(P_Tf)^2). \]

\end{prop}

\begin{proof}
For $t\in[0,T]$, define the functional
\[
\Phi(t) = a(t) P_{t} \left( \Gamma (P_{T-t} f)\right)
	+ b(t) P_{t} \left( \Gamma^Z (P_{T-t} f)\right),
\]
where $a,b:[0,T]\rightarrow[0,\infty)$ are control functions yet to be chosen.
A straightforward computation shows that
\begin{align*}
\Phi'(t)
	&= a'(t) P_t \left( \Gamma (P_{T-t} f)\right) + b'(t) P_{t}
		\left( \Gamma^Z (P_{T-t} f)\right)\\
	&\qquad + 2a(t) P_{t} \left( \Gamma_2 (P_{T-t} f)\right)+2b(t)P_{t}
		\left( \Gamma_2^Z (P_{T-t} f)\right),
\end{align*}
and the inequality (\ref{e.curv}) implies that
\[
\Gamma_2 (P_{T-t} f) + \frac{b(t)}{a(t)} \Gamma_2^Z (P_{T-t} f)
	\ge - \frac{\beta a(t)}{b(t)} \Gamma(P_{T-t} f) + \alpha  \Gamma^Z
		(P_{T-t} f).
\]
Thus,
\begin{equation}
\label{e.22}
\Phi'\ge \left(a'- 2\beta \frac{a^2}{b} \right)P_t(\Gamma (P_{T-t} f))
	+ (b'+2\alpha a) P_t(\Gamma^Z (P_{T-t} f)).
\end{equation}
We now choose the functions $a$ and $b$ so that
\[
b'+2\alpha a = 0
\]
and
\[
a' - 2\beta \frac{a^2}{b} = C,
\]
where $C$ is a constant independent of $t$. This leads to the  candidates
\[
a(t)=\frac{1}{\alpha} (T -t)
\]
and
\[
b(t)=(T -t)^2.
\]
For this choice of $a$ and $b$, the inequality (\ref{e.22}) becomes
\begin{equation}
\label{e.phip}
\Phi'(t) \ge  -\frac{1}{\alpha} \left( 1+\frac{2\beta}{\alpha} \right)
	P_{t} \left( \Gamma (P_{T-t} f) \right).
\end{equation}
By Lemma \ref{l.simple} with $\phi(x) = x^2$, we have that
\[ \frac{d}{dt} P_t(P_{T-t}f)^2 = P_t(\Gamma(P_{T-t}f)), \]
and thus integrating \eqref{e.phip} from 0 to $T$ yields the desired result.
\end{proof}

Under an additional assumption, the curvature bound (\ref{e.curv}) also implies
a reverse log Sobolev type inequality which is much stronger than the
previous reverse
Poincar\'e inequality.  The additional required assumption here
is the following commutation: for any $f\in C^\infty(G)$,
\begin{equation}
\label{commutation_gamma}
\Gamma( f, \Gamma^Z(f))=\Gamma^Z( f , \Gamma(f)).
\end{equation}
First we prove the following lemma given this assumption.

\begin{lem}
\label{l.Phi}
For fixed $T>0$, $x\in G$, and $f\in \mathcal{C}$, define the entropy functionals
\[
\Phi_1 (t) = P_t \left( (P_{T-t} f) \Gamma (\ln P_{T-t}f) \right)(x)
\]
and
\[
\Phi_2 (t) = P_t \left( (P_{T-t} f) \Gamma^Z (\ln P_{T-t}f) \right)(x),
\]
for $t\in[0,T]$.  Then, assuming \eqref{commutation_gamma}
holds,
\[
\Phi'_1 (t) = 2P_t \left( (P_{T-t} f) \Gamma_2 (\ln P_{T-t}f) \right)(x)
\]
and
\[
\Phi'_2 (t) = 2P_t \left( (P_{T-t} f) \Gamma_2^Z (\ln P_{T-t}f) \right)(x).
\]
\end{lem}

\begin{proof}
For $t\in[0,T]$ and $x\in G$, consider the functionals
\[
\phi_1 (t,x)=(P_{T-t} f)(x) \Gamma (\ln P_{T-t}f)(x)
\]
and
\[
\phi_2 (t,x)= (P_{T-t} f)(x) \Gamma^Z (\ln P_{T-t}f)(x).
\]
If we prove that
\[
L\phi_1+\frac{\partial\phi_1}{\partial t} = 2 (P_{T-t} f) \Gamma_2 (\ln P_{T-t}f)
\]
and
\[
L\phi_2+\frac{\partial \phi_2}{\partial t} = 2 (P_{T-t} f) \Gamma_2^Z (\ln P_{T-t}f),
\]
then the result follows almost immediately.

Again for simplicity, take $u(t,x) = P_{T-t} f(x)$ and here let
$u_t=\frac{\partial u}{\partial t}$.  Then a simple computation gives
\[
\frac{\partial \phi_1}{\partial t} = u_t \Gamma(\ln u) + 2 u \Gamma\left(\ln
	u,\frac{u_t}{u}\right).
\]
On the other hand,
\[
L\phi_1 = Lu \Gamma(\ln u) + u L \Gamma(\ln u) + 2 \Gamma(u,\Gamma(\ln u)).
\]
Combining these equations we obtain
\[
L\phi_1 + \frac{\p \phi_1}{\p t}
	= u L\Gamma(\ln u) +  2\Gamma(u,\Gamma(\ln u)) + 2 u \Gamma\left(\ln
		u,\frac{u_t}{u}\right).
\]
We now see that
\begin{align*}
2 u \Gamma_2(\ln u)
	&= u (L \Gamma(\ln u) - 2 \Gamma(\ln u,L(\ln u))) \\
	&= u L\Gamma(\ln u) - 2 u \Gamma(\ln u,L(\ln u)).
\end{align*}
Observing that
\[
L(\ln u) = - \frac{\Gamma(u)}{u^2} - \frac{u_t}{u},
\]
we may conclude
\[
L\phi_1+\frac{\partial \phi_1}{\partial t} =2 (P_{T-t} f) \Gamma_2 (\ln P_{T-t}f).
\]

In the same vein, we obtain
\[
L\phi_2 + \frac{\partial \phi_2}{\partial t}
	= u L\Gamma^Z(\ln u) + 2\Gamma(u,\Gamma^Z(\ln u))
		+ 2 u \Gamma^Z\left(\ln u,\frac{u_t}{u}\right).
\]
This time using the definition of $\Gamma^Z_2 $, we find
\begin{align*}
2 u \Gamma_2^Z(\ln u)
	&= u (L \Gamma^Z(\ln u) - 2 \Gamma^Z(\ln u,L(\ln u))) \\
	&= u L\Gamma^Z(\ln u) + 2 u \Gamma^Z\left(\ln
		u,\frac{\Gamma(u)}{u^2}\right) + 2 u
		\Gamma^Z\left(\ln u,\frac{u_t}{u}\right).
\end{align*}
From this last equation it is now clear that under the assumption
\eqref{commutation_gamma} we have
\[
L\phi_2 + \frac{\p \phi_2}{\p t} =  2 u \Gamma_2^Z(\ln u),
\]
and this concludes the proof.
\end{proof}

Given this lemma, we may now prove the following reverse log Sobolev
inequalities holds.

\begin{thm}
\label{P:linearizedBL}
Suppose that (\ref{e.curv}) and (\ref{commutation_gamma}) are satisfied. Then,
for any $T>0$ and $f \in \mathcal{C}^+$,
\[
\Gamma( \ln P_T f) + \alpha T \Gamma^Z (\ln P_T f)
	\le \frac{ 1+\frac{2\beta}{\alpha} }{T}
		\left(\frac{P_T (f \ln f)}{P_Tf} - \ln P_T f \right).
 \]
In particular, the following reverse log Sobolev inequality holds
\begin{equation}
\label{e.logsob}
\Gamma( \ln P_T f)
	\le  \frac{ 1+\frac{2\beta}{\alpha} }{T}
		\left(\frac{P_T (f \ln f)}{ P_T f} - \ln P_T f \right).
\end{equation}
\end{thm}

\begin{proof}
For $t\in[0,T]$, we define the functional
\[
\Psi(t)=a(t) \Phi_1 (t) +b(t) \Phi_2 (t),
\]
where $\Phi_1$ and $\Phi_2$ are as defined in Lemma \ref{l.Phi} and
$a,b$ are non-negative control functions to be chosen later.  Since we assume
\eqref{commutation_gamma} holds, Lemma
\ref{l.Phi} implies that
\begin{multline*}
\Psi'(t) = a'(t) \Phi_1 (t) + b'(t) \Phi_2 (t) \\
	+ 2a(t)P_t \left( (P_{T-t} f) \Gamma_2 (\ln P_{T-t}f) \right)
	+ 2b(t)P_t \left( (P_{T-t} f) \Gamma^Z_2 (\ln P_{T-t}f) \right).
\end{multline*}
Thus, given the curvature bound and working exactly as in Proposition
\ref{p.rpoin}, we are lead to the same choices
\[
a(t)=\frac{1}{\alpha} (T -t)
	\quad \text{ and } \quad b(t)=(T -t)^2.
\]
For this choice of $a$ and $b$, we have the inequality
\begin{equation}
\label{e.psip}
\Psi'(t)  \ge -\frac{1}{\alpha} \left( 1+\frac{2\beta}{\alpha} \right)  \Phi_1 (t).
\end{equation}
Taking $\phi(x) = x\ln x$ in Lemma \ref{l.simple} implies that
\begin{align*}
\frac{d}{dt} P_t( (P_{T-t}f)(\ln P_{T-t}f))
	&= P_t\left(\frac{\Gamma(P_{T-t}f)}{P_{T-t}f}\right) \\
	&= P_t( (P_{T-t}f)\Gamma(\ln P_{T-t}f))
	=  \Phi_1(t),
\end{align*}
and thus integrating \eqref{e.psip} from $0$ to $T$ then yields the claimed result.
\end{proof}

\begin{rem}
Other choices of control functions could be made to satisfy the desired
criteria in the proofs of Proposition \ref{p.rpoin} and Theorem
\ref{P:linearizedBL}.  In particular, we could have
taken
\[ a(t) = \frac{1}{\alpha}( (1+\delta)T-t) \]
and
\[ b(t) = ( (1+\delta)T-t)^2, \]
for any $\delta\ge0$.  This choice of $a$ and $b$ would give the following generalized
estimates.  The following would generalize Proposition \ref{p.rpoin}:
for all $T>0$ and $f\in L^\infty(G)$,
\begin{multline*}
\Gamma(P_Tf) + \alpha (1+\delta) T\Gamma^Z(P_Tf) \\
	\le \frac{1 + \frac{2\beta}{\alpha}}{(1+\delta)T}(P_T(f^2)-(P_Tf)^2)
		+ \frac{\delta}{1+\delta} P_T(\Gamma(f)) + \frac{ \alpha
		\delta^2}{1+\delta}T P_T(\Gamma^Z(f)).
\end{multline*}
Also, the following statement would generalize Theorem \ref{P:linearizedBL}:
for all $T>0$ and $f\in \mathcal{C}$,
\begin{align*}
P_T f \Gamma( &\ln P_T f) + \alpha (1+\delta) T P_Tf \Gamma^Z (\ln P_T f)  \\
	&\le \frac{ 1+\frac{2\beta}{\alpha} }{(1+\delta)T}  \left(P_T (f \ln f)
		- (P_T f) \ln P_T f \right) \\
	&\qquad + \frac{\delta}{1+\delta} P_T \left(  f \Gamma (\ln f) \right)
		+  \frac{ \alpha \delta^2}{1+\delta}T P_T \left(  f \Gamma^Z (\ln f) \right).
\end{align*}
\end{rem}

\subsection{Wang type and integrated Harnack inequalities}
\label{s.finhar}

A reverse log Sobolev inequality such as in Theorem \ref{P:linearizedBL}
is sufficient to prove an analogue of Wang's dimension-free Harnack inequality.  Estimates of this type
were first proved by Wang in a Riemannian setting under the assumption of a
lower bound on the Ricci curvature \cite{Wang1997a}.  Before stating the
estimate, we must make the following definition.

\begin{nota}
(Horizontal distance)
\label{n.finlength}

\begin{enumerate}
\item The {\em length} of a $C^1$-path $\sigma:[a,b]\rightarrow
G$ is defined as
\[ \ell(\sigma)
	:= \int_a^b |L_{\sigma^{-1}(s)*}\dot{\sigma}(s)|_{\mathfrak{g}} \,ds.
\]

\item \label{i.fin2}
A $C^1$-path $\sigma:[a,b]\rightarrow G$ is {\em horizontal} if
$L_{\sigma(t)^{-1}*}\dot{\sigma}(t)\in \mathcal{H}\times\{0\}$
for a.e.~$t$.  Let $C^{1,h}$ denote the set of horizontal paths
$\sigma:[0,1]\rightarrow G$.

\item The {\em horizontal distance} between $x,y\in G$ is defined by
\[ d(x,y) := \inf\{\ell(\sigma): \sigma\in C^{1,h} \text{ such
    that } \sigma(0)=x \text{ and } \sigma(1)=y \}. \]
\end{enumerate}
\end{nota}

\begin{prop}
\label{p.har}
Suppose there exists a constant $C<\infty$ such that, for all  $T>0$ and $f\in
\mathcal{C}^+$,
\begin{equation}
\label{e.rls}
\Gamma( \ln P_T f)
	\le \frac{C}{T} \left(\frac{P_T (f \ln f)}{P_Tf} - \ln P_T f \right).
\end{equation}
Then, for all $T>0$, $x,y\in G$, $f\in L^\infty(G)$ with $f\ge0$, and $p\in(1,\infty)$,
\begin{equation}
\label{e.wang}
( P_T f )^p (x) \le P_T f^p (y)\exp\left(C\frac{d^2(x,y) }{4(p-1)T } \right).
\end{equation}
\end{prop}

\begin{proof}
First take $f\in\mathcal{C}^+$.
Let $b(s)=1+(p-1)s$ for $s\in[0,1]$ and $\sigma:[0,1] \rightarrow G$
be an arbitrary horizontal path such that $\sigma(0)=x$ and
$\sigma(1)=y$. Define the functional
\[
\phi (s)=\frac{p}{ b  (s)}  \ln P_T f^{b(s)} (\sigma(s)),
\quad\text{ for } s\in[0,1].
\]
Since
\begin{align*}
\frac{d}{ds} f^{b(s)}(\sigma(s))
	&=  f^{b(s)}(\sigma(s))\left( (p-1)\ln f(\sigma(s)) +
		b(s)\langle(d(\ln f))(\sigma(s)),\sigma'(s)\rangle\right),
\end{align*}
differentiating $\phi$ with respect to $s$ and applying (\ref{e.rls}) yields
\begin{align*}
\phi' (s)
	&= -\frac{p(p-1)}{b(s)^2}\ln P_Tf^{b(s)} \\	
	&\qquad + \frac{p}{b(s)}
		\left( \frac{(p-1)P_T (f^{b(s)}\ln f) + b(s) P_T(f^{b(s)}\langle d(\ln
		f),\sigma'\rangle)}{P_Tf^{b(s)}}\right) \\
	&= \frac{p (p-1)}{b(s)^2}\left( \frac{P_T (f^{b(s)} \ln f^{b(s)})}{P_T f^{b(s)} }
		- \ln P_T f^{b(s)} \right)  \\
	&\qquad +\frac{p}{ b  (s)} \langle d(\ln P_T f^{b(s)}) , \sigma'(s) \rangle \\
	&\ge \frac{p (p-1) T }{b(s)^2 C } \Gamma(\ln  P_T
		f^{b(s)})+\frac{p}{ b  (s)} \langle d(\ln P_T f^{b(s)}) ,
		\sigma'(s) \rangle.
\end{align*}
Now, for every $\lambda>0$,
\[
 \langle d(\ln P_T f^{b(s)}) , \sigma'(s) \rangle
	\ge -\frac{1}{2\lambda} \Gamma(\ln  P_T f^{b(s)})-\frac{\lambda}{2}
		|\sigma'(s)|^2,
\]
since $\sigma$ horizontal implies that
$\sigma'(s)\in\mathrm{span}\{\tilde{X}_i(\sigma(s))\}_{i=1}^n$.  In particular, choosing
\[
\lambda=\frac{C}{ 2p (p-1) T}b(s)^2
\]
gives
\[
\phi' (s) \ge -\frac{C}{4p(p-1)T}(1+(p -1)s)^2  | \sigma'(s) |^2.
\]
Integrating this inequality from $0$ to $1$ yields
\[
\ln P_T f^p (y) -\ln (P_T f)^p (x)
	\ge - \frac{C}{4p(p-1)T}\int_0^1
		(1+(p -1)s)^2  | \sigma'(s) |^2 ds .
\]
Maximizing $\int_0^1 (1+(p-1)s)^2  | \sigma'(s)|^2 ds$ over the set of
horizontal paths such that $\sigma(0)=x$ and
$\sigma(1)=y$ shows that (\ref{e.wang}) holds for $f\in\mathcal{C}\cap C^\infty(G)$.

To prove the estimate for general $f\ge0$, let $C^\infty(G)\ni h_n\ge0$ be an increasing sequence of
functions with compact support such that $h_n\uparrow1$.  Then (\ref{e.wang})
holds for $g=h_nP_\tau
f+\varepsilon\in\mathcal{C}\cap C^\infty(G)$ for all $n$, $\tau>0$, and
$\varepsilon>0$.  Then letting $\varepsilon\rightarrow0$, $\tau\rightarrow0$,
and $n\rightarrow\infty$ in the inequality completes the proof of (\ref{e.wang}) for $f\ge0$.
\end{proof}

Wang type Harnack inequalities in turn are equivalent to so-called
``integrated Harnack inequalities'' as in the following lemma.  Here we follow
the proof of 2.4 in \cite{Wang2010a}.  An alternative form and proof of the
following equivalence can be found in Lemma D.1 of \cite{DriverGordina2009}.

\begin{lem}
\label{l.harqi2}
Let $T>0$, $x,y\in G$, $p\in(1,\infty)$, and $C\in(0,\infty]$.  Then
\begin{equation}
\label{e.wang1}
(P_Tf)^p(x) \le  C P_Tf^p(y), \quad\text{ for all } f\in L^\infty(G) \text{ with
} f\ge0,
\end{equation}
if and only if
\begin{equation}
\label{e.a}
\left(\int_G \left[\frac{p_T(x,z)}{p_T(y,z)}\right]^{1/(p-1)}
	p_T(y,z)\,dz \right)^{p-1} \le C.
\end{equation}
\end{lem}

\begin{proof}
Set $J_{x,y}(z)=\frac{p_T(x,z)}{p_T(y,z)}$ and $f_n:=(n\wedge J_{x,y})^{1/(p-1)}$ for
$n\ge1$.  Then applying \eqref{e.wang1} to $f_n$ yields
\begin{align*}
(P_T f_n)^p(x)
	&\le C P_Tf^p_n(y)
	= C \int_G (n\wedge J_{x,y}(z))^{p/(p-1)} p_T(y,z)\,dz \\
	&\le C\int_G \left(n \wedge
		\frac{p_T(x,z)}{p_T(y,z)}\right)^{1/(p-1)}\,p_T(x,z)\,dz
	= CP_Tf_n(x).
\end{align*}
Thus,
\[ P_TJ_{x,y}^{1/(p-1)}(x) = \lim_{n\rightarrow\infty} P_Tf_n(x) \le
	C^{1/(p-1)}, \]
which yields \eqref{e.a}.

For the converse, we have by H\"older's inequality that
\begin{align*}
P_Tf(x) &= \int_G f(z) \frac{p_T(x,z)}{p_T(y,z)} p_T(y,z)dz \\
	&\le (P_Tf^p)^{1/p}(y) \left(\int_G \frac{p_T(x,z)}{p_T(y,z)}^{p/(p-1)} p_T(y,z)dz
		\right)^{(p-1)/p} \\
	&= (P_Tf^p)^{1/p}(y) \left(\int_G \frac{p_T(x,z)}{p_T(y,z)}^{1/(p-1)}
		p_T(x,z)dz \right)^{(p-1)/p} \\
	&\le (P_Tf^p)^{1/p}(y) C^{1/p},
\end{align*}
which completes the proof.
\end{proof}

\section{Infinite-dimensional Heisenberg-like groups}
\label{s.infheis}

In this section, we review the definitions for infinite-dimensional
Heisenberg-like groups, which are infinite-dimensional Lie groups based on an
abstract Wiener space.  Much of the material is this section also appears in
\cite{GordinaMelcher2011}.

\subsection{Abstract Wiener spaces}
\label{s.wiener}
For the reader's convenience, we summarize several well-known properties of
Gaussian measures and abstract Wiener spaces that are required for
the sequel.  These results as well as more details on abstract Wiener spaces
and some particular examples may be found in \cite{Bogachev1998,Kuo1975}.

Suppose that $W$ is a real separable Banach space and $\mathcal{B}_{W}$ is
the Borel $\sigma$-algebra on $W$.

\begin{df}
\label{d.2.1}
A measure $\mu$ on $(W,\mathcal{B}_{W})$ is called a (mean zero,
non-degenerate) {\it Gaussian measure} provided that its characteristic
functional is given by
\begin{equation}
\label{e.gauss}
\hat{\mu}(u) := \int_W e^{iu(x)} d\mu(x)
	= e^{-\frac{1}{2}q(u,u)}, \qquad \text{ for all } u\in W^*,
\end{equation}
for $q=q_\mu:W^*\times W^*\rightarrow\mathbb{R}$ a symmetric, positive
definite quadratic form.
That is, $q$ is a real inner product on $W^*$.
\end{df}

\begin{lem}
\label{l.q}
If $u,v\in W^*$, then
\[ \int_W u(w)v(w)\,d\mu(w) = q(u,v). \]
\end{lem}

\begin{proof}
Let $u_*\mu:=\mu\circ u^{-1}$ denote the measure on $\mathbb{R}$ which is the
push forward of $\mu$ under $u$.
Then by equation (\ref{e.gauss}) $u_*\mu$ is normal with mean 0 and variance
$q(u,u)$.  Thus,
\[ \int_W u^2(w)\,d\mu(w) = q(u,u). \]
Polarizing this identity gives the desired result.
\end{proof}

A proof of the following standard theorem may be found for example in Appendix
A of \cite{DriverGordina2008}.

\begin{thm}
\label{t.2.3}
Let $\mu$ be a Gaussian measure on a real separable Banach space $W$.
For $p\in[1,\infty)$, let
\begin{equation}
\label{e.2.2}
C_p :=\int_W \|w\| _{W}^{p} \,d\mu(w).
\end{equation}
For $w\in W$, let
\[
\|w\|_H := \sup\limits_{u\in W^*\setminus\{0\}}\frac{|u(w)|}{\sqrt{q(u,u)}}
\]
and define the {\em Cameron-Martin subspace} $H\subset W$ by
\[ H := \{h\in W : \|h\|_H < \infty\}. \]
Then
\begin{enumerate}
\item \label{i.1}
For all  $p\in[1,\infty)$, $C_p<\infty$.

\item $H$ is a dense subspace of $W$.

\item There exists a unique inner product $\langle\cdot,\cdot\rangle_H$
on $H$ such that $\|h\|_H^2 = \langle h,h\rangle_H$ for all $h\in H$, and
$H$ is a separable Hilbert space with respect to this inner product.

\item \label{i.3}
For any $h\in H$,
$\|h\|_W \le \sqrt{C_2} \|h\|_H$.

\item \label{i.5}
If $\{e_j\}_{j=1}^\infty$ is an orthonormal basis for $H$, then for any
$u,v\in H^*$
\[ q(u,v) = \langle u,v\rangle_{H^*} = \sum_{j=1}^\infty u(e_j)v(e_j).
\]
\end{enumerate}
\end{thm}

It follows from item (\ref{i.3}) that any $u\in W^*$ restricted to $H$ is in
$H^*$.  Therefore, by item (\ref{i.5}) and Lemma \ref{l.q},
\begin{equation}
\label{e.bl}
\int_W u^2(w)\,d\mu(w) = q(u,u) = \|u\|_{H^*}^2 = \sum_{j=1}^\infty |u(e_j)|^2.
\end{equation}
More generally we have the following lemma.

\begin{lem}
\label{l.lin}
Let $K$ be a real Hilbert space and $\varphi:W\rightarrow K$ be a linear map.  Then
\[ \|\varphi\|_{H^*\otimes K}^2
	= \sum_{j=1}^\infty \|\varphi(e_j)\|_K^2
	= \int_W \|\varphi(w)\|_K^2\,d\mu(w).\]
\end{lem}

\begin{proof}
Let $\{f_\ell\}_{\ell=1}^{\mathrm{dim}(K)}$ be an orthonormal basis of $K$.
Then by equation (\ref{e.bl})
\begin{align*}
\int_W \|\varphi(w)\|_K^2 \,d\mu(w)
	&= \int_W \sum_{\ell=1}^{\mathrm{dim}(K)} |\langle
		\varphi(w),f_\ell\rangle_K|^2\,d\mu(w) \\
	&= \sum_{\ell=1}^{\mathrm{dim}(K)} \|\langle\varphi(\cdot),f_\ell\rangle_K\|_{H^*}^2
	= \sum_{\ell=1}^{\mathrm{dim}(K)} \sum_{j=1}^\infty
		|\langle\varphi(e_j),f_\ell\rangle_K|^2 \\
	&= \sum_{j=1}^\infty \|\varphi(e_j)\|_K^2
	= \|\varphi\|_{H^*\otimes K}^2.
\end{align*}
\end{proof}

This leads to the following facts for linear maps on $W$.  First we set the
some notation.

\begin{nota}
\label{n.HS}
Let $K$ be a real Hilbert space, and suppose
$\alpha:H^{\otimes m}\rightarrow K$ is a multi-linear map.  Then the
Hilbert-Schmidt norm of $\alpha$ is defined by
\begin{align*}
\|\alpha\|_2^2
	:= \|\alpha\|_{(H^*)^{\otimes m}\otimes K}^2
	&= \sum_{j_1,\ldots,j_m=1}^\infty
		\|\alpha(e_{j_1},\ldots,e_{j_m})\|_K^2 \\
	&= \sum_{j_1,\ldots,j_m=1}^\infty \sum_{\ell=1}^{\mathrm{dim}(K)}
		\langle\alpha(e_{j_1},\ldots,e_{j_m}),f_\ell\rangle_K^2,
\end{align*}
where $\{e_j\}_{j=1}^\infty$ and $\{f_\ell\}_{\ell=1}^{\mathrm{dim}(K)}$ are
orthonormal bases of $H$ and $K$, respectively.
\end{nota}

One may verify directly that these norms are independent of the chosen bases.

\begin{lem}
\label{l.1}
Suppose $K$ is a Hilbert space and $\varphi:W\rightarrow K$ is a continuous
linear map.  Then $\varphi:H\rightarrow K$ is Hilbert-Schmidt, that
is, $\|\varphi\|_2<\infty$.
\end{lem}

\begin{proof}
By Lemma \ref{l.lin},
\begin{align*}
\|\varphi\|_2^2
	= \|\varphi\|^2_{H^*\otimes K}
	&= \int_W \|\varphi(w)\|_K^2\,d\mu(w) \\
	&\le \|\varphi\|_0^2 \int_W \|w\|_W^2 \,d\mu(w)
	= C_2\|\varphi\|_0^2,
\end{align*}
where $C_2<\infty$ is as defined in (\ref{e.2.2}) and
\[ \|\varphi\|_0 := \sup\{\varphi(w)\|_K : \|w\|_W=1\} <\infty \]
by the continuity of $\varphi$.
\end{proof}

Similarly, we may prove the following.
\begin{lem}
\label{l.2}
Suppose $K$ is a Hilbert space and $\rho:W\times W\rightarrow K$ is a
continuous bilinear map.
Then $\rho:H\times H\rightarrow K$ is Hilbert-Schmidt.
\end{lem}

\begin{proof}
Note first that, for each $w\in W$, $\varphi=\rho(w,\cdot)$ is a
continuous linear operator and thus, by the proof Lemma \ref{l.1},
\[ \|\rho(w,\cdot)\|_2^2 = \|\rho(w,\cdot)\|_{H^*\otimes K}
	\le  C_2\|\rho(w,\cdot)\|_0^2
	\le C_2\|\rho\|_0\|w\|_W^2,\]
where
\[ \|\rho\|_0 := \sup\{\rho(w,w')\|_K : \|w\|_W=\|w'\|_W=1\} <\infty. \]
Then viewing $w\mapsto\rho(w,\cdot)$ as a continuous linear map from
$W$ to the Hilbert space $H^*\otimes K$, Lemma \ref{l.lin} implies that
\begin{align*}
\|\rho\|_2^2
	&= \|h\mapsto\rho(h,\cdot)\|_{H^*\otimes(H^*\otimes K)}^2
	= \int_W \|\rho(w,\cdot)\|_{H^*\otimes K}^2\,d\mu(w) \\
	&\le \int_W C_2\|\rho\|_0^2\|w\|_W^2\,d\mu(w)
	= C_2^2\|\rho\|_0^2<\infty.
\end{align*}
\end{proof}

\subsection{Infinite-dimensional Heisenberg-like groups}
\label{s.3}

We revisit the definition of the infinite-dimensional
Heisenberg-like groups that were first considered in \cite{DriverGordina2008}.
Note that since we are interested in subelliptic heat kernel measures
on these groups, there are some necessary modifications to the
topology as was done in \cite{GordinaMelcher2011}.  First we set
the following notation which will hold for the rest of the paper.

\begin{nota}
Let $(W,H,\mu)$ be a real abstract Wiener space.  Let $\mathbf{C}$
be a real Hilbert space with inner product
$\langle\cdot,\cdot\rangle_\mathbf{C}$ and
$\mathrm{dim}(\mathbf{C})=d<\infty$.  Let $\omega:W\times
W\rightarrow\mathbf{C}$  be a continuous skew-symmetric bilinear
form on $W$.  We will also trivially assume that $\omega$
is surjective (otherwise, we just restrict to a linear subspace of
$\mathbf{C}$).
\end{nota}

\begin{df}
Let $\mathfrak{g}$ denote $W\times\mathbf{C}$ when thought of as a Lie algebra
with the Lie bracket given by
\begin{equation}
\label{e.3.5}
[(X_1,V_1), (X_2,V_2)] := (0, \omega(X_1,X_2)).
\end{equation}
Let $G$ denote $W\times\mathbf{C}$ when thought of as a group with
multiplication given by
\begin{equation}
\label{e.3.1}
 g_1 g_2 := g_1 + g_2 + \frac{1}{2}[g_1,g_2],
\end{equation}
where $g_1$ and $g_2$ are viewed as elements of $\mathfrak{g}$. For $g_i=(w_i,c_i)$, this may be written equivalently as
\begin{equation}
\label{e.3.2}
(w_1,c_1)\cdot(w_2,c_2) = \left( w_1 + w_2, c_1 + c_2 +
    \frac{1}{2}\omega(w_1,w_2)\right).
\end{equation}
We will call $G$ constructed in this way a {\em Heisenberg-like group}.
\end{df}
It is easy to verify that, given this bracket and multiplication,
$\mathfrak{g}$ is indeed a Lie algebra and $G$ is a group.
Note that $g^{-1}=-g$ and the identity $e=(0,0)$.

\begin{nota}
Let $\mathfrak{g}_{CM}$ denote $H\times\mathbf{C}$ when thought of as a Lie
subalgebra of $\mathfrak{g}$, and we will refer to $\mathfrak{g}_{CM}$ as the
{\em Cameron-Martin subalgebra} of $\mathfrak{g}$. Similarly, let $G_{CM}$
denote $H\times\mathbf{C}$ when thought of as a subgroup of $G$, and we will
refer to $G_{CM}$ as the {\em Cameron-Martin subgroup} of $G$.
\end{nota}
We will equip $\mathfrak{g}=G$ with the homogeneous norm
\[ \|(w,c)\|_{\mathfrak{g}} := \sqrt{\|w\|_W^2 + \|c\|_\mathbf{C}}, \]
and analogously on $\mathfrak{g}_{CM}=G_{CM}$ we define
\[ \|(A,a)\|_{\mathfrak{g}_{CM}} := \sqrt{\|A\|_H^2 + \|a\|_\mathbf{C}}. \]

One may easily see that $G$ and $G_{CM}$ are topological groups with respect
to the topologies induced by the homogeneous norms, see for example Lemma 2.9
of \cite{GordinaMelcher2011}.

Before proceeding, let us give the basic motivating examples for the construction of
these infinite-dimensional Heisenberg-like groups.  In what follows, if $X$ is
a complex vector space, let $X_{\mathrm{Re}}$ denote $X$ thought of as a real
vector space.  If $(H,\langle\cdot,\cdot\rangle_H)$ is a complex Hilbert space, let
$\langle\cdot,\cdot\rangle_{H_{\mathrm{Re}}}:=\mathrm{Re}\langle\cdot,\cdot\rangle_H$,
in which case $(H_{\mathrm{Re}}, \langle\cdot,\cdot\rangle_{H_{\mathrm{Re}}})$
becomes a real Hilbert space.
\begin{example} [Finite-dimensional Heisenberg group]
\label{ex.Heis}
Let $W=H=(\mathbb{C}^n)_\mathrm{Re}\cong\mathbb{R}^{2n}$
and $\mu$ be standard Gaussian measure on
$\mathbb{R}^{2n}$.  Then $(W,H,\mu)$ is an abstract Wiener space.  Let
$\mathbf{C}=\mathbb{R}$ and $\omega(w,z):= \mathrm{Im}\langle w,z\rangle$,
where $\langle w,z\rangle=w\cdot\bar{z}$ is the usual inner product on
$\mathbb{C}^n$.
Then $G=\mathbb{R}^{2n}\times\mathbb{R}$ equipped with a group operation as
defined in (\ref{e.3.2}) is a finite-dimensional
Heisenberg group.
\end{example}

\begin{example}
[Heisenberg group of a symplectic vector space]
\label{ex.infHeis}
Let $(K,\langle\cdot,\cdot\rangle)$ be a complex Hilbert space and $Q$ be a
strictly positive trace class operator on $K$.  For $h,k\in K$, let $\langle
h,k\rangle_Q:= \langle h,Qk\rangle$ and $\|h\|_Q:= \sqrt{\langle
h,h\rangle_Q}$, and let $(K_Q,\langle\cdot,\cdot\rangle_Q)$ denote the Hilbert
space completion of $(K,\|\cdot\|_Q)$.
Then $W=(K_Q)_\mathrm{Re}$ and $H=K_\mathrm{Re}$ determines an abstract Wiener
space (see, for example, exercise 17 on p.59 of \cite{Kuo1975}).  Letting $\mathbf{C}=\mathbb{R}$ and
\[ \omega( w,z ) := \mathrm{Im}\langle w,z\rangle_Q, \]
then $G=(K_Q)_\mathrm{Re}\times\mathbb{R}$ equipped with a group operation as
defined in (\ref{e.3.2}) is an infinite-dimensional
Heisenberg-like group.
\end{example}

\subsection{Finite-dimensional projection groups}
\label{s.gpproj}
The finite-dimensional projections of $G$ defined in this section
will be important in the sequel.  Note that the construction of these
projections is quite natural in the sense that they come from the usual projections of the
abstract Wiener space; however, the projections defined here are not group
homomorphisms, which is a complicating factor in the analysis.

As usual, let $(W,H,\mu)$ denote an abstract Wiener space.
Let $i:H\rightarrow W$ be the inclusion map, and $i^*:W^*\rightarrow H^*$ be
its transpose so that $i^*\ell:=\ell\circ i$ for all $\ell\in W^*$.  Also,
let
\[ H_* := \{h\in H: \langle\cdot,h\rangle_H\in \mathrm{Range}(i^*)\subset H^*\}.
\]
That is, for $h\in H$, $h\in H_*$ if and only if $\langle\cdot,h\rangle_H\in
H^*$ extends to a continuous linear functional on $W$, which we will continue
to denote by $\langle\cdot,h\rangle_H$.  Because $H$ is a dense subspace of
$W$, $i^*$ is injective and thus has a dense range.  Since
$H\ni h\mapsto\langle\cdot,h\rangle_H\in H^*$ is a
linear isometric isomorphism, it follows that $H_*\ni
h\mapsto\langle\cdot,h\rangle_H\in W^*$ is a linear isomorphism
also, and so $H_*$ is a dense subspace of $H$.

Suppose that $P:H\rightarrow H$ is a finite rank orthogonal projection
such that $PH\subset H_*$.  Let $\{ e_j\}_{j=1}^n$ be an orthonormal basis for
$PH$.  Then we may extend $P$ to a (unique) continuous operator
from $W\rightarrow H$ (still denoted by $P$) by letting
\begin{equation}
\label{e.proj}
Pw := \sum_{j=1}^n \langle w, e_j\rangle_H  e_j
\end{equation}
for all $w\in W$.

\begin{nota}
\label{n.proj}
Let $\mathrm{Proj}(W)$ denote the collection of finite rank projections
on $W$ such that
\begin{enumerate}
\item $PW\subset H_*$,
\item $P|_H:H\rightarrow H$ is an orthogonal projection (that is, $P$ has the form given in equation
\eqref{e.proj}), and
\item $PW$ is sufficiently large to satisfy
H\"ormander's condition (that is, $\{\omega(A,B):A,B\in PW\}=\mathbf{C}$).
\end{enumerate}
For each $P\in\mathrm{Proj}(W)$, we define $G_P:=
PW\times\mathbf{C}\subset H_*\times\mathbf{C}$ and a corresponding
projection $\pi_P:G\rightarrow G_P$ \[ \pi_P(w,x):= (Pw,x). \] We
will also let $\mathfrak{g}_P=\mathrm{Lie}(G_P) = PW\times\mathbf{C}$.
In the context of Section 2, note that, for each $P\in\mathrm{Proj}(W)$, $G_P$
is a finite-dimensional connected unimodular Lie group (in fact, $\mathfrak{g}_P$ is step
2 stratified) with $\mathcal{H}=PH$ and $\mathcal{V}=\mathbf{C}$.
\end{nota}

\subsection{Derivatives and differential forms on $G$}
\label{s.deriv}

For $x\in G$, again let $L_x:G\rightarrow G$ and $R_x:G\rightarrow G$
denote left and right multiplication by $x$, respectively.  As $G$
is a vector space, to each $x\in G$ we can associate the tangent
space $T_x G$ to $G$ at $x$, which is naturally isomorphic to $G$.

\begin{nota}[Linear and group derivatives]
\label{n.3.5}
Let $f:G\rightarrow\mathbb{C}$ denote a Frech\'{e}t smooth function for $G$
considered as a Banach space with respect to the norm
\[
|(w,c)|_G := \sqrt{\|w\|_W^2+\|c\|_\mathbf{C}^2}.
\]
Then, for $x\in G$, and $h,k\in\mathfrak{g}$, let
\[ f'(x)h := \partial_h f(x) = \frac{d}{dt}\bigg|_0f(x+th) \]
and
\[
f''(x)  \left(  h\otimes k\right)  :=\partial
_{h}\partial_{k}f(x).
\]
For $v,x\in G$, let $v_x \in T_x G$ denote the tangent vector
satisfying $v_xf=f'(x)v$.  If $x(t)$ is any smooth curve in
$G$ such that $x(0) = x$ and $\dot{x}(0)=v$ (for example,
$x(t) = x+tv$), then
\[ L_{g*} v_x = \frac{d}{dt}\bigg|_0 g\cdot x(t). \]
In particular, for $x=e$ and $v_e=h\in\mathfrak{g}$, again we let
$\tilde{h}(g):=L_{g*}h$, so that $\tilde{h}$ is the unique left invariant
vector field on $G$ such that $\tilde{h}(e)=h$.  As usual we view
$\tilde{h}$ as a first order differential operator acting on smooth
functions by
\[ (\tilde{h}f)(x) = \frac{d}{dt}\bigg|_0 f(x\cdot \sigma(t)), \]
where $\sigma(t)$ is a smooth curve in $G$ such that $\sigma(0)=e$ and
$\dot{\sigma}(0)=h$ (for example, $\sigma(t)=th$).
\end{nota}

\begin{prop}
\label{p.3.7}
Let $f:G\rightarrow\mathbb{R}$ be a smooth function,
$h=(A,a)\in\mathfrak{g}$ and $x=\left(  w,c\right)  \in G.$ Then
\[
\widetilde{h}(x)  :=l_{x\ast}h=\left(  A,a+\frac{1}{2}%
\omega\left(  w,A\right)  \right)  _{x}\text{ for all } x=\left(  w,c\right)
\in G
\]
and in particular
\begin{equation}
\widetilde{\left(  A,a\right)  }f(x)  =f'\left(
x\right)  \left(  A,a+\frac{1}{2}\omega\left(  w,A\right)  \right)  .
\label{e.3.12}%
\end{equation}
Furthermore, if $h=\left(  A,a\right)  ,$ $k=\left(  B,b\right)$, then
\begin{equation}
\left(  \tilde{h}\tilde{k}f-\tilde{k}\tilde{h}f\right)  =\widetilde{\left[
h,k\right]  }f. \label{e.3.13}
\end{equation}
That is, the Lie algebra structure on $\mathfrak{g}$ induced by the Lie
algebra structure on the left invariant vector fields on $G$ is the same as
the Lie algebra structure defined in equation \eqref{e.3.5}.
\end{prop}

\begin{proof}
Since $th=t\left(  A,a\right)  $ is a curve in $G$ passing through the
identity at $t=0,$ we have
\begin{align*}
\widetilde{h}(x)   &  =\frac{d}{dt}\bigg|_{0}\left[  x\cdot\left(
th\right)  \right]  =\frac{d}{dt}\bigg|_{0}\left[  \left(  w,c\right)  \cdot
t(A,a)\right] \\
&  =\frac{d}{dt}\bigg|_{0}\left[  \left(  w+tA,c+ta+\frac{t}{2}%
\omega\left(  w,A\right)  \right)  \right] \\
&  =\left(  A,a+\frac{1}{2}\omega\left(  w,A\right)  \right)  .
\end{align*}
So by the chain rule, $(  \tilde{h}f)  (x)
=f'(x)  \tilde{h}(x)  $ and hence
\begin{align}
(  \tilde{h}\tilde{k}f)  (x)   &  =\frac{d}
{dt}\bigg|_{0}\left[  f'\left(  x\cdot th\right)  \tilde{k}\left(
x\cdot th\right)  \right] \nonumber\\
&  =f''(x)  \left(  \tilde{h}\left(
x\right) \otimes\tilde{k}(x)  \right)+f''\left(
x\right)
\frac{d}{dt}\bigg|_{0}\tilde{k}\left(  x\cdot th\right), \label{e.3.14}%
\end{align}
where%
\[
\frac{d}{dt}\bigg|_{0}\tilde{k}\left(  x\cdot th\right)  =\frac{d}{dt}%
\bigg|_{0}\left(  B,a+\frac{1}{2}\omega\left(  w+tA,B\right)  \right)
=\left( 0,\frac{1}{2}\omega\left(  A, B\right)  \right).
\]
Since $f''(x)  $ is symmetric, it now follows by
subtracting equation \eqref{e.3.14} from itself with $h$ and $k$ interchanged
that
\[
\left(  \tilde{h}\tilde{k}f-\tilde{k}\tilde{h}f\right)  (x)
=f'(x)  \left(  0,\omega\left(  A,B\right)  \right)
=f'(x)  \left[  h,k\right]  =\left(  \widetilde{\left[
h,k\right]  }f\right)  (x)
\]
as desired.
\end{proof}

\begin{df}
\label{d.cyl}
A function $f:G\rightarrow\mathbb{C}$ is a {\em (smooth) cylinder function}
if it may be written as $f=F\circ\pi_P$, for some $P\in\mathrm{Proj}(W)$
and (smooth) $F:G_P\rightarrow\mathbb{C}$.
\end{df}

\begin{lem}
\label{l.3.15}
Suppose that $\ell:H\rightarrow\mathbf{C}$ is a
continuous linear map. Then for any orthonormal basis $\left\{  e_{j}\right\}
_{j=1}^{\infty}$ of $H$ the series
\begin{equation}
\sum_{j=1}^{\infty}\ell\left(  e_{j}\right)  \otimes\ell\left(  e_{j}\right)
\in\mathbf{C}\otimes\mathbf{C} \label{e.3.33}%
\end{equation}
and
\begin{equation}
\sum_{j=1}^{\infty}\ell\left(  e_{j}\right)  \otimes e_{j}\in\mathbf{C}\otimes
H \label{e.3.34}%
\end{equation}
are convergent and independent of the basis.
\end{lem}

\begin{proof}
First note that by Lemma \ref{l.1}, $\ell$ is Hilbert-Schmidt, and thus
\[
\sum_{j=1}^{\infty} \left\| \ell\left(  e_{j}\right)  \otimes\ell\left(
e_{j}\right)  \right\| _{\mathbf{C}\otimes\mathbf{C}}
	=\sum_{j=1}^{\infty}\left\| \ell\left(  e_{j}\right)  \right\| _{\mathbf{C}
}^{2} <\infty, \]
which shows that the sum in equation \eqref{e.3.33} is absolutely convergent.
Similarly, since $\left\{  \ell\left(
e_{j}\right)  \otimes e_{j}\right\}  _{j=1}^{\infty}$ is an orthogonal set in
$\mathbf{C}\otimes H$ and
\[
\sum_{j=1}^{\infty}\left\| \ell\left(  e_{j}\right)  \otimes e_{j}
\right\| _{\mathbf{C}\otimes H}^{2}
	=\sum_{j=1}^{\infty}\left\| \ell\left(  e_{j}\right)  \right\| _{\mathbf{C}}^{2}<\infty,
\]
the sum in equation \eqref{e.3.34} is convergent as well.

Now recall that if $H$ and $K$ are two real Hilbert spaces then the Hilbert space
tensor product $H\otimes K$ is unitarily equivalent to the space of
Hilbert-Schmidt operators $HS\left(  H,K\right)$ from $H$ to $K$. Under
this identification, $h\otimes k\in H\otimes K$ corresponds to the operator
(still denoted by $h\otimes k)$ in $HS\left(  H,K\right)  $ defined by
\[ H\ni
h'\mapsto (h\otimes k)h'=\left(  h,h'\right)  _{H}\,k\in K. \]
Using this identification we have that, for all $c\in\mathbf{C}$,
\begin{align*}
\left(  \sum_{j=1}^{\infty}\ell\left(  e_{j}\right)  \otimes\ell\left(
e_{j}\right)  \right)  c  &  =\sum_{j=1}^{\infty}\ell\left(  e_{j}\right)
\left\langle \ell\left(  e_{j}\right)  ,c\right\rangle _{\mathbf{C}}%
=\sum_{j=1}^{\infty}\ell\left(  e_{j}\right)  \left\langle e_{j},\ell^{\ast
}c\right\rangle _{\mathbf{C}}\\
&  =\ell\left(  \sum_{j=1}^{\infty}\left\langle e_{j},\ell^{\ast
}c\right\rangle _{\mathbf{C}}e_{j}\right)  =\ell\ell^{\ast}c
\end{align*}
and%
\[
\left(  \sum_{j=1}^{\infty}\ell\left(  e_{j}\right)  \otimes e_{j}\right)
c=\sum_{j=1}^{\infty}e_{j}\left\langle \ell\left(  e_{j}\right)
,c\right\rangle _{\mathbf{C}}=\sum_{j=1}^{\infty}e_{j}\left\langle e_{j}%
,\ell^{\ast}c\right\rangle _{\mathbf{C}}=\ell^{\ast}c,
\]
which clearly shows that equations \eqref{e.3.33} and\eqref{e.3.34} are basis-independent.
\end{proof}

\begin{nota}
\label{n.3.26}For $x=\left(  w,c\right)  \in G,$ let $\gamma(x)
$ and $\chi(x)  $ be the elements of $\mathfrak{g}_{CM}%
\otimes\mathfrak{g}_{CM}$ defined by
\begin{align*}
\gamma(x)   &  :=\sum_{j=1}^{\infty}\left(  0,\omega\left(
w,e_{j}\right)  \right)  \otimes\left(  e_{j},0\right)  \text{ and}\\
\chi(x)   &  :=\sum_{j=1}^{\infty}\left(  0,\omega\left(
w,e_{j}\right)  \right)  \otimes\left(  0,\omega\left(  w,e_{j}\right)
\right)
\end{align*}
where $\left\{  e_{j}\right\}  _{j=1}^{\infty}$ is any orthonormal basis for
$H.$ Both $\gamma$ and $\chi$ are well defined by Lemma \ref{l.3.15}
taking $\ell=\omega(x,\cdot)$.
\end{nota}

The following proposition is proved in Proposition 3.29 of
\cite{DriverGordina2008}, although the statement given there is for the elliptic
case.  We reproduce the short proof here for the reader's convenience.

\begin{prop}
\label{p.3.28}
Let $\left\{  e_{j}\right\}_{j=1}^{\infty}$ be an orthonormal
basis for $H$. Then, for any smooth cylinder function $f:G\rightarrow
\mathbb{R}$,
\[Lf(x)  :=\sum_{j=1}^{\infty}\left[  \widetilde{\left(
e_{j},0\right)  }^{2}f\right]  (x)
\]
is well-defined and independent of basis. In particular, if $f=F\circ\pi_{P}$,
$x=(w,c)\in G$, $\partial_{h}$ is as in
Notation \ref{n.3.5} for all $h\in\mathfrak{g}_{CM}$, and
\[
\Delta_{H}f(x)  :=\sum_{j=1}^{\infty}\partial_{\left(
e_{j},0\right)  }^{2}f(x)  =\left(  \Delta_{PH}F\right)  \left(
Pw,c\right),
\]
then
\begin{equation}
\label{e.Lconv}
Lf(x)  =\left(  \Delta_{H}f\right)  \left(
x\right)  +f''(x)  \left(  \gamma(x)
+\frac{1}{4}\chi(x)  \right),
\end{equation}
where $\gamma$ and $\chi$ are as defined in Notation \ref{n.3.26}.
\end{prop}

\begin{proof}
Recall from equation \eqref{e.3.12} that
\[ \widetilde{\left(  e_{j},0\right)  }f(x)  =f'\left(
x\right)  \left(  e_{j},\frac{1}{2}\omega\left(  w,e_{j}\right)  \right)  .
\]
Applying $\widetilde{\left(  e_{j},0\right)  }$ to both sides of this equation
then gives
\begin{align*}
\widetilde{\left(  e_{j},0\right)  }^{2}f(x)   &  =f^{\prime
\prime}(x)  \left(  \left(  e_{j},\frac{1}{2}\omega\left(
w,e_{j}\right)  \right)  \otimes\left(  e_{j},\frac{1}{2}\omega\left(
w,e_{j}\right)  \right)  \right) \\
&  =f''(x)  \left(  \left(  e_{j},0\right)
\otimes\left(  e_{j},0\right)  \right)  +f''(x)
\left(  \left(  0,\omega\left(  w,e_{j}\right)  \right)  \otimes\left(
e_{j},0\right)  \right) \nonumber\\
	&\quad+\frac{1}{4}f''(x)  \left(  \left(
0,\omega\left(  w,e_{j}\right)  \right)  \otimes\left(  0,\omega\left(
w,e_{j}\right)  \right)  \right),
\end{align*}
wherein we have used that
\[
\partial_{e_{j}}\omega\left(  \cdot,e_{j}\right)  =\omega\left(  e_{j}%
,e_{j}\right)  =0
\]
and the fact that $f''(x)$ is symmetric.  Summing on
$j$ then shows that
\begin{align*}
\sum_{j=1}^{\infty}\left[  \widetilde{\left(  e_{j},0\right)  }^{2}f\right]
(x)   &  =\sum_{j=1}^{\infty}f''(x)
\left(  \left(  e_{j},0\right)  \otimes\left(  e_{j},0\right)  \right)
+f''(x)  \left(  \gamma(x)  +\frac
{1}{4}\chi(x)  \right) \\
&  =\sum_{j=1}^{\infty}\partial_{\left(  e_{j},0\right)  }^{2}f\left(
x\right)  +f''(x)  \left(  \gamma(x)
+\frac{1}{4}\chi(x)  \right),
\end{align*}
which verifies equation (\ref{e.Lconv}) and thus shows that $Lf$ is indendent of
the choice of orthonormal basis for $H$.
\end{proof}

Similarly, we may prove the following proposition.

\begin{prop}
Let $\left\{  e_{j}\right\}_{j=1}^{\infty}$ be an orthonormal
basis for $H$. Then, for any smooth cylinder functions $f,g:G\rightarrow
\mathbb{R}$,
\begin{equation*}
\Gamma(f,g)\left(  x\right)
	:= \sum_{j=1}^{\infty}
		\left(  \widetilde{\left(
		e_{j},0\right)  }f\right) \left(  x\right)\left(\widetilde{\left(
		e_{j},0\right)  }g\right)  \left(  x\right)
\end{equation*}
is well-defined and independent of basis.  In particular, if $f=F\circ\pi_{P}$
and $g=G\circ\pi_Q$ for $P,Q\in\mathrm{Proj}(W)$, $x=(w,c)\in G$,
$\partial_{h}$ is as in
Notation \ref{n.3.5} for all $h\in\mathfrak{g}_{CM}$, and
\[ \nabla_H f(x) := \sum_{j=1}^\infty (\partial_{(e_j,0)}f(x)) e_j
	= \nabla_{PH}F(Pw,c), \]
then
\begin{multline*}
\Gamma(f,g)(x) = \langle\nabla_Hf(x),\nabla_H g(x)\rangle_H + \frac{1}{4}(f'(x)\otimes g'(x))
		\chi(x) \\
	+ \frac{1}{2}(f'(x)\otimes g'(x) + g'(x)\otimes f'(x))\gamma(x),
\end{multline*}
where $\gamma$ and $\chi$ are as defined in Notation \ref{n.3.26}.
\end{prop}

\begin{proof}
Recall again from equation \eqref{e.3.12} that\begin{align*}
\widetilde{(e_j,0)}f(x)
	&= f'(x)\left(e_j,\frac{1}{2}\omega(w,e_j)\right)
	= f'(x) \left( (e_j,0) + \frac{1}{2}(0,\omega(w,e_j)\right) \\
	&= \partial_{(e_j,0)} f(x) + \frac{1}{2}f'(x)(0,\omega(w,e_j)).
\end{align*}
Thus,
\begin{multline*}
\Gamma(f,g)(x)
	=  \sum_{j=1}^\infty \bigg\{\partial_{(e_j,0)} f(x)\partial_{(e_j,0)} g(x)
		+ \frac{1}{2}f'(x)(0,e_j) g'(x)(0,\omega(w,e_j)) \\
	+ \frac{1}{2}f'(x)(0,\omega(w,e_j))g'(x)(0,e_j)
		+ \frac{1}{4}f'(x)(0,\omega(w,e_j))g'(x)(0,\omega(w,e_j))\bigg\}.
\end{multline*}
Note for example that
\begin{align*}
\sum_{j=1}^\infty f'(x)(0,&\omega(w,e_j)) g'(x)(e_j,0)) \\
	&= \sum_{j=1}^\infty \langle f'(x)\otimes g'(x),
		(0,\omega(w,e_j))\otimes(e_j,0)\rangle \\
	&= \left\langle f'(x)\otimes g'(x), \sum_{j=1}^\infty
		(0,\omega(w,e_j))\otimes(e_j,0)\right\rangle \\
	&= \langle f'(x)\otimes g'(x), \gamma(x)\rangle.
\end{align*}
Similarly,
\begin{align*}
\sum_{j=1}^\infty f'(x)(0,\omega(w,e_j)&)g'(x)(0,\omega(w,e_j)) \\
	&= \left\langle f'(x)\otimes g'(x), \sum_{j=1}^\infty
		(0,\omega(w,e_j))\otimes(0,\omega(w,e_j))\right\rangle \\
	&= \langle f'(x)\otimes g'(x),\chi(x) \rangle.
\end{align*}
\end{proof}

Thus, along with $L$ and $\Gamma$, we are able to consider the following differential
forms which are well-defined for smooth cylinder functions $f$, $g$ on $G$
\begin{align*}
& \Gamma_{2}\left( f, g \right):=\frac{1}{2}\left( L\Gamma\left( f, g
\right)-\Gamma\left( f, Lg\right)-\Gamma\left( g, Lf\right)\right), \\
& \Gamma^{Z}\left( f, g \right):=\sum_{\ell=1}^{d} \left(\widetilde{\left( 0,
f_{\ell}\right) }f \right) \left(\widetilde{\left( 0, f_{\ell}\right) }g \right),\text{
and}
 \\
& \Gamma_2^Z(f,g) := \frac{1}{2} \left(L\Gamma^Z(f,g) - \Gamma^Z(f,Lg) -
\Gamma^Z(g,Lf)\right).
\end{align*}
Of course, for the finite-dimensional groups $G_P$ we may define the same
forms for $f,g\in C^\infty(G_P)$ as was done for more general finite-dimensional
groups in Section \ref{s.finite}.  These
will be denoted by $L_P$, $\Gamma_P$, $\Gamma_{2,P}$, and $\Gamma_{2,P}^Z$.  In
particular, if $\{e_i\}_{i=1}^n$ is an orthonormal basis of $PH$, then
\[ L_Pf = \sum_{j=1}^n \widetilde{(e_j,0)}^2f
	\quad\text{ and }\quad
\Gamma_P(f,g) = \sum_{j=1}^n
	\left(\widetilde{(e_j,0)}f\right)\left(\widetilde{(e_j,0)}g\right). \]

\subsection{Distances on $G$}
\label{s.length}

We define the sub-Riemannian distance on $G_{CM}$
analogously to how it was done in finite dimensions in Section
\ref{s.finite}.  We recall its relevant properties, including the
fact that the topology induced by this metric is equivalent to the
topology induced by $\|\cdot\|_{\mathfrak{g}_{CM}}$.

\begin{nota}
(Horizontal distance on $G_{CM}$)
\label{n.length}

\begin{enumerate}
\item For $x=(A,a)\in G_{CM}$, let
\[ |x|_{\mathfrak{g}_{CM}}^2 : = \|A\|_H^2 + \|a\|_\mathbf{C}^2. \]
The {\em length} of a $C^1$-path $\sigma:[a,b]\rightarrow
G_{CM}$ is defined as
\[ \ell(\sigma)
	:= \int_a^b |L_{\sigma^{-1}(s)*}\dot{\sigma}(s)|_{\mathfrak{g}_{CM}} \,ds.
\]

\item \label{i.2}
A $C^1$-path $\sigma:[a,b]\rightarrow G_{CM}$ is {\em horizontal} if
$L_{\sigma(t)^{-1}*}\dot{\sigma}(t)\in H\times\{0\}$
for a.e.~$t$.  Let $C^{1,h}_{CM}$ denote the set of horizontal paths
$\sigma:[0,1]\rightarrow G_{CM}$.

\item The {\em horizontal distance} between $x,y\in G_{CM}$ is defined by
\[ d(x,y) := \inf\{\ell(\sigma): \sigma\in C^{1,h}_{CM} \text{ such
    that } \sigma(0)=x \text{ and } \sigma(1)=y \}. \]
\end{enumerate}
The horizontal distance is defined analogously on $G_P$ and will be denoted by
$d_P$.  In particular, for a sequence
$\{P_n\}_{n=1}^\infty\subset\mathrm{Proj}(W)$, we will let $d_n:=d_{P_n}$.
\end{nota}

\begin{rem}
\label{r.horiz}
Note that if $\sigma(t)=(A(t),a(t))$ is a horizontal path, then
\[ L_{\sigma(t)^{-1}*}\dot{\sigma}(t)
	= \left(\dot{A}(t), \dot{a}(t) - \frac{1}{2}\omega(A(t),\dot{A}(t))\right)
		\in H\times\{0\}
\]
implies that $\sigma$ must satisfy
\[ a(t) = a(0) + \frac{1}{2}\int_0^t \omega(A(s),\dot{A}(s))\,ds, \]
and the length of $\sigma$ is given by
\begin{align*}
\ell(\sigma)
	= \int_0^1 |L_{\sigma^{-1}(s)*}\dot{\sigma}(s)|_{\mathfrak{g}_{CM}}\,ds
	= \int_0^1 \|\dot{A}(s)\|_H\,ds .
\end{align*}
\end{rem}

The following proposition is Propositions 2.17 and 2.18 of
\cite{GordinaMelcher2011}.  We
refer the reader to that paper for the proof.
\begin{prop}
\label{p.length}
If $\{\omega(A,B): A,B\in H\}=\mathbf{C}$,
then there exist finite constants $K_1=K_1(\omega)$ and $K_2=K_2(d,\omega)$ such that
\begin{equation}
\label{e.length}
K_1(\|A\|_H+\sqrt{\|a\|_\mathbf{C}}) \le d(e,(A,a))
	\le K_2(\|A\|_H+\sqrt{\|a\|_\mathbf{C}}),
\end{equation}
for all $(A,a)\in\mathfrak{g}_{CM}$.  In particular, this is sufficient to
imply that the topologies induced by $d$ and $\|\cdot\|_{\mathfrak{g}_{CM}}$
are equivalent.
\end{prop}

\begin{rem}
The equivalence of the homogeneous norm and horizontal distance
topologies is a standard result in finite dimensions.  However, the
usual proof of this result relies on compactness arguments that
must be avoided in infinite dimensions.  Thus, the proof for
Proposition \ref{p.length} included in \cite{GordinaMelcher2011}
necessarily relies on different methods particular to the structure
of the present groups.  Using these methods, we are currently unable to remove the
dependence on $d=\mathrm{dim}(\mathbf{C})$ from the coefficient in
the upper bound. The reader is referred to \cite{GordinaMelcher2011}
for further details.

\end{rem}

The following fact is not required for the sequel.  However, it is natural to
expect and thus we include it for completeness.  Also, a similar limiting
argument will be employed in the proof of quasi-invariance in Theorem
\ref{t.qi}.

\begin{lem}
\label{l.dn}
For any $m\in\mathbb{N}$ and $x \in G_m$,
\[
d_{n}\left( e, x \right) \rightarrow d\left( e, x \right), \text{ as }
n\rightarrow\infty. \]
\end{lem}

\begin{proof}
First it is clear that, for any $n$ and $x,y\in G_n$, $d_{n}\left( x, y \right) \ge d\left( x, y
\right)$. In particular, if $x, y \in G_{m}$ for some $m$, then $x, y \in
G_{n}$ for all $n \ge m $ and $d_{n}\left( x, y \right)$ is decreasing
as $n \to \infty$. Now let $x=(w,c)\in G_m$, and consider an arbitrary horizontal path
$\sigma:[0,1]\rightarrow G_{CM}$ such that $\sigma(0)=e$ and
$\sigma(1)=g$.  Recall that, by Remark \ref{r.horiz}, $\sigma$ must have the
form
\[ \sigma(t) = \left(A(t),\frac{1}{2}\int_0^t \omega(A(s),\dot{A}(s))\,ds
	\right). \]
For $n\ge m$, consider the ``projected'' horizontal paths
$\sigma_n:[0,1]\rightarrow G_n$ given by
\[ \sigma_n(t)
	= (A_n(t),a_n(t))
	:= \left(P_nA(t),\frac{1}{2}\int_0^t \omega(P_nA(s),P_n\dot{A}(s))
		\,ds\right). \]
Note that $A_n(1)=P_nA(1)=P_nw=w$, and let
\[ \varepsilon_n
	:= c- a_n(1)
	= c-\frac{1}{2}\int_0^1 \omega(P_nA(s),P_n\dot{A}(s))\,ds\in\mathbf{C}. \]
Then, for $d_n$ the horizontal distance in $G_n$,
\[ d_n(e,x) = d_n(e,(w,c))
	= d_n(e,(w,a_n(1)+\varepsilon_n))
	= d_n(e,(w,a_n(1))\cdot(0,\varepsilon_n)). \]
Now, for any left-invariant metric $d$, we have that
\[
d(e,xy) \le d(e,x) + d(x,xy) = d(e,x) + d(e,y).
\]
Thus,
\begin{equation}
\label{e.dn}
d_n(e,x) \le d_n(e,(w,a_n(1))) + d_n(e,(0,\varepsilon_n)) \\
	\le \ell(\sigma_n) + C\sqrt{\|\varepsilon_n\|_\mathbf{C}},
\end{equation}
where the second inequality
holds by \eqref{e.length} with constant $C=C(d,\omega)$ not depending on
$n$.  (Of course, the estimate \eqref{e.length} is stated for the horizontal
distance $d$ on $G_{CM}$ and not
$d_n$.  However, it is clear from the proof in
\cite{GordinaMelcher2011} that the same estimate holds for each
$d_P$  with common coefficients $K_1$ and $K_2$ for all sufficiently
large $P\in\mathrm{Proj}(W)$.  See the proof of Proposition 2.17
of \cite{GordinaMelcher2011} for details.)

Now, for any $k\ge n$, it is clear that $\ell(\sigma_n)\le\ell(\sigma_k)$,
since
\[
\ell(\sigma_n) = \int_0^1 \|P_n\dot{A}(s)\|_H\,ds
	= \int_0^1 \sqrt{\sum_{j=1}^n \left|\langle
		\dot{A}(s),e_j\rangle_H\right|^2}\, ds,
\]
where $\{e_j\}_{j=1}^n$ is an orthonormal basis of $P_nH$.  Thus, for all $k\ge n$,
\begin{equation}
\label{e.11}
d_n(e,x)\le \ell(\sigma_k) + C\sqrt{\|\varepsilon_n\|_\mathbf{C}}.
\end{equation}
Dominated convergence implies that
\[ \lim_{k\rightarrow\infty} \ell(\sigma_k)
	= \lim_{k\rightarrow\infty} \int_0^1 \|P_k\dot{A}(s)\|\,ds
	= \int_0^1 \|\dot{A}(s)\|\,ds
	= \ell(\sigma), \]
and thus allowing $k\rightarrow\infty$ in (\ref{e.11}) gives
\[ d_n(e,x)\le \ell(\sigma) + C\sqrt{\|\varepsilon_n\|_\mathbf{C}}. \]
Now taking the infimum over all horizontal paths in $G_{CM}$ such that $\sigma(0)=e$ and
$\sigma(1)=g$ implies that
\[ d_n(e,x)\le d(e,x) + C\sqrt{\|\varepsilon_n\|_\mathbf{C}}. \]
One may also show via dominated convergence that
\[ \lim_{n\rightarrow\infty} \|\varepsilon_n\|_\mathbf{C}
	= \lim_{n\rightarrow\infty}
		\left\|\frac{1}{2}\int_0^1 \omega(A(s),\dot{A}(s)) -
		\omega(P_nA(s),P_n\dot{A}(s))\,ds \right\|_\mathbf{C}
	= 0.
\]
Thus, given an arbitrary $\varepsilon>0$, for all sufficiently large $n$,
\[
d(e,x) \le d_n(e,x) \le d\left( e, x \right) + \varepsilon.
\]
\end{proof}

\section{Infinite-dimensional computations}
Now given the structure of the infinite-dimensional Heisenberg-like groups and
their finite-dimensional projections defined in the previous section,
we wish to consider estimates like the ones discussed in Section
\ref{s.finite}.  In the first subsection, we will show that the
desired curvature-dimension estimate (\ref{e.curv}) and commutation formula
(\ref{commutation_gamma}) hold for the differential forms defined
on $G$ and $G_P$.  In the second section, we then record the reverse
inequalities and Harnack estimates that follow as a result.

\subsection{Curvature-dimension bounds and commutation relations}
\label{s.Pcurv}

In this section, $\left\{ e_{i}\right\}_{i=1}^{\infty}$ and
$\left\{ f_{\ell}\right\}  _{\ell=1}^{d}$ will denote orthonormal bases
for $H$ and $\mathbf{C}$ respectively, where $d=\operatorname{dim}
\mathbf{C}$.
Let $\|\omega\|_2$ denote the Hilbert-Schmidt norm of $\omega:H\times
H\rightarrow\mathbf{C}$ as defined in Notation \ref{n.HS}.  That is,
\[
\left\| \omega\right\| _{2}^{2}
	:= \left\| \omega\right\| _{H^{\ast}\otimes H^{\ast}\otimes\mathbf{C}}
	:= \sum_{i,j=1}^{\infty}\left\| \omega\left(  e_{i},e_{j}\right)  \right\| _{\mathbf{C}}^{2}
	= \sum_{i,j=1}^\infty \sum_{\ell=1}^d \langle\omega(e_i,e_j),f_\ell\rangle_\mathbf{C}^2.
\]
As $\omega$ is a continuous bilinear operator on $W$, Lemma \ref{l.2} implies
that $\omega$ is Hilbert-Schmidt and so $\|\omega\|_2<\infty$.

For the rest of this section, we also fix $P\in\operatorname*{Proj}\left(
W\right)$ and let $\{e_i\}_{i=1}^n$ denote an orthonormal basis of $PH$.  Let
$\|\omega\|_{2,P}$ denote the Hilbert-Schmidt norm of $\omega$ restricted to
$PH$, that is,
\[
\| \omega \|_{2, P}^{2}
	:= \sum_{i, j=1}^{n} \|\omega \left( e_{i}, e_{j} \right)\|_{\mathbf{C}}^{2}
	= \sum_{i, j=1}^{n} \sum_{\ell=1}^{d}
		\langle \omega \left( e_{i}, e_{j} \right), f_{\ell}\rangle_{\mathbf{C}}^{2}.
\]
We will also let
\begin{align*}
\rho_{2} &:= \inf \left\{\sum_{i, j=1}^{\infty}\left(\sum_{\ell=1}^{d} \langle
\omega\left( e_{i}, e_{j}\right), f_{\ell}
\rangle_{\mathbf{C}}x_{\ell}\right)^{2}:\sum_{\ell=1}^{d}  x_{\ell}
^{2}=1 \right\} \text{ and }
\\
\rho_{2, P} &:= \inf \left\{\sum_{i, j=1}^{n}\left(\sum_{\ell=1}^{d} \langle \omega\left( e_{i}, e_{j}\right), f_{\ell} \rangle_{\mathbf{C}}x_{\ell}\right)^{2}:\sum_{\ell=1}^{d} x_{\ell}^{2}=1 \right\}.
\end{align*}
It is clear that
\[ 0<\rho_{2} \le \| \omega \|_{2}^{2} \quad \text{ and } \quad
	0<\rho_{2, P} \le \| \omega \|_{2, P}^{2}.
\]

First we need the following computational lemmas.

\begin{lem}
\label{l.4.1}
For any smooth cylinder function $f$,
\[
\rho_{2}\Gamma^{Z}\left( f \right)
	\le \sum_{i, j=1}^{\infty} \left( \widetilde{ \left( 0, \omega \left( e_{i}, e_{j} \right)\right)}f
	\right)^{2}\le \| \omega \|_{2}^{2}\Gamma^{Z}\left( f \right).
\]
Similarly, for any $f\in C^\infty(G_P)$,
\[
\rho_{2, P}\Gamma^{Z}\left( f \right)
	\le \sum_{i, j=1}^{n} \left( \widetilde{
		\left( 0, \omega \left( e_{i}, e_{j} \right)\right)}f \right)^{2}\le \|
		\omega \|_{2, P}^{2}\Gamma^{Z}\left( f \right).
\]
\end{lem}

\begin{proof}
We will prove only the first set of inequalities, as the second proof is obviously
parallel.  The upper bound follows from the Cauchy-Schwarz inequality, since
\begin{align*}
\sum_{i, j=1}^{\infty}
	&\left( \widetilde{ \left( 0, \omega \left( e_{i}, e_{j} \right)\right)}f \right)^{2}
	=\sum_{i, j=1}^{\infty} \left(\sum_{\ell=1}^{d}\langle \omega \left(
		e_{i}, e_{j} \right), f_{\ell}\rangle_{\mathbf{C}} \widetilde{ \left( 0,
		f_{\ell}\right)}f \right)^{2} \\
	&\le \sum_{i, j=1}^{\infty} \left(\sum_{\ell=1}^{d}\langle \omega \left(
		e_{i}, e_{j} \right), f_{\ell}\rangle_{\mathbf{C}}^{2} \right)
		\left( \sum_{\ell=1}^{d}
		\left(\widetilde{ \left( 0, f_{\ell}\right)}f \right)^{2}\right)
	=\| \omega \|_{2}^{2}\Gamma^{Z}\left( f \right).
\end{align*}
To see the lower bound, simply note that
\[
\sum_{i, j=1}^{\infty} \left( \widetilde{ \left( 0, \omega \left( e_{i}, e_{j}
		\right)\right)}f \right)^{2}
	=\sum_{i, j=1}^{\infty} \left(\sum_{\ell=1}^{d}\langle \omega \left( e_{i}, e_{j} \right),
		f_{\ell}\rangle_{\mathbf{C}} \widetilde{ \left( 0, f_{\ell}\right)}f \right)^{2}
	\ge \rho_{2}\Gamma^{Z}\left( f \right).
\]
\end{proof}

\begin{lem}
\label{l.gz2}
For any smooth cylinder function $f$,
\[ \Gamma^{Z}_{2}\left( f \right)
	= \sum_{j=1}^{\infty} \sum_{\ell=1}^{d}\left(
		\widetilde{\left(e_{j},0\right)} \widetilde{ \left( 0, f_{\ell}\right)}f
		\right)^{2} \text{ and } \]
\[
\sum_{j=1}^{\infty} \left(
\sum_{i=1}^{\infty}\widetilde{\left(e_{i},0\right)}\widetilde{ \left( 0,
\omega \left( e_{i}, e_{j} \right)\right)}f \right)^{2} \le \| \omega
\|_{2}^{2} \Gamma^{Z}_{2}\left( f \right).
\]
Similarly, for any $f\in C^\infty(G_P)$,
\[ \Gamma^{Z}_{2,P}\left( f \right)
	= \sum_{j=1}^n \sum_{\ell=1}^{d}\left(
		\widetilde{\left(e_{j},0\right)} \widetilde{ \left( 0, f_{\ell}\right)}f
		\right)^{2} \text{ and } \]
\[
\sum_{j=1}^n \left(
\sum_{i=1}^n \widetilde{\left(e_{i},0\right)}\widetilde{ \left( 0,
\omega \left( e_{i}, e_{j} \right)\right)}f \right)^{2}
	\le \| \omega\|_{2,P}^{2} \Gamma^{Z}_{2,P}\left( f \right).
\]
\end{lem}

\begin{proof}
We find $\Gamma^{Z}_{2}\left( f \right)$ by a straightforward calculation:
\begin{align*}
\Gamma^{Z}_{2}\left( f \right)
	&=\frac{1}{2}\sum_{j=1}^{\infty}\widetilde{
		\left(e_{j},0\right)}^2\sum_{\ell=1}^{d} \left(\widetilde{ \left( 0,
f_{\ell}\right)}f\right)^{2}
	- \sum_{\ell=1}^{d} \sum_{j=1}^{\infty} \left(
\widetilde{ \left( 0, f_{\ell}\right)}f\right) \left( \widetilde{ \left(
0, f_{\ell}\right)}\widetilde{ \left(e_{j},0\right)}^2 f\right)
\\
	&= \sum_{j=1}^{\infty} \sum_{\ell=1}^{d} \left\{
\left(\widetilde{(e_j,0)}\widetilde{(0,f_\ell)}f\right)^2 +
		\left(\widetilde{(0,f_\ell)}f\right)
		\left(\widetilde{(e_j,0)}^2\widetilde{(0,f_\ell)}f\right) \right\} \\
	&\qquad - \sum_{\ell=1}^{d} \sum_{j=1}^{\infty} \left(
\widetilde{ \left( 0, f_{\ell}\right)}f\right) \left( \widetilde{ \left(
0, f_{\ell}\right)}\widetilde{ \left(e_{j},0\right)}^2 f\right)
\\
	&= \sum_{j=1}^{\infty} \sum_{\ell=1}^{d}\left( \widetilde{\left(e_{j},0\right)} \widetilde{ \left( 0, f_{\ell}\right)}f \right)^{2},
\end{align*}
where we have used that $\widetilde{ \left( 0, f_{\ell}\right)}$ and $\widetilde{ \left(
e_{j},0\right)}$ commute by equation (\ref{e.3.13}).  The Cauchy-Schwarz inequality
then implies that
\begin{align*}
\sum_{j=1}^{\infty} &\left(
\sum_{i=1}^{\infty}\widetilde{\left(e_{i},0\right)}\widetilde{ \left( 0,
\omega \left( e_{i}, e_{j} \right)\right)}f \right)^{2}\\
	&=\sum_{j=1}^{\infty} \left(\sum_{i=1}^{\infty} \sum_{\ell=1}^{d}\langle
		\omega \left( e_{i}, e_{j} \right),
		f_{\ell}\rangle_{\mathbf{C}}\widetilde{\left(e_{i},0\right)} \widetilde{
		\left( 0, f_{\ell}\right)}f\right)^{2} \\
& \le\sum_{j=1}^{\infty}\left(\sum_{i=1}^{\infty} \sum_{\ell=1}^{d}\langle \omega \left( e_{i}, e_{j} \right), f_{\ell}\rangle_{\mathbf{C}}^{2}\right) \left(\sum_{i=1}^{\infty} \sum_{\ell=1}^{d}\left(\widetilde{\left(e_{i},0\right)}\widetilde{ \left( 0, f_{\ell}\right)}f\right)^{2}\right)
\\
	&= \| \omega \|_{2}^{2}\Gamma^{Z}_{2}\left( f \right) .
\end{align*}
\end{proof}

Lemmas \ref{l.4.1} and \ref{l.gz2} combine to give us the desired
curvature-dimension
bound.

\begin{prop}
\label{p.curvature}
For any $\nu >0$ and smooth cylinder function $f$,
\begin{equation}
\label{curvature}
\Gamma_{2}\left( f \right)+\nu \Gamma_{2}^{Z}\left( f \right)
	\ge \rho_{2}\Gamma^{Z}\left( f \right) -\frac{\| \omega \|_{2}^{2}
	}{\nu}\Gamma\left( f \right).
\end{equation}
Similarly, for any $\nu >0$ and $f\in C^\infty(G_P)$,
\begin{equation}
\Gamma_{2,P}\left( f \right)+\nu \Gamma_{2,P}^{Z}\left( f \right)
	\ge \rho_{2,P}\Gamma_P^{Z}\left( f \right) -\frac{\| \omega \|_{2,P}^{2}
	}{\nu}\Gamma_P\left( f \right).
\end{equation}

\end{prop}

\begin{proof}
Again, we prove only the first inequality.
To do this, we first use the commutation relation in \eqref{e.3.5} and the estimates
in Lemma \ref{l.4.1} to estimate $\sum_{i,
j=1}^{\infty}
\left(\widetilde{\left(e_{i},0\right)}\widetilde{\left(e_{j},0\right)}f
\right)^{2}$. In what follows, we also use the antisymmetry of the form
$\omega$.
\begin{align*}
& \sum_{i, j=1}^{\infty}
\left(\widetilde{\left(e_{i},0\right)}\widetilde{\left(e_{j},0\right)}f
\right)^{2} \\
&=\sum_{i, j=1}^{\infty} \left(\frac{\widetilde{\left(e_{i},0\right)}\widetilde{\left(e_{j},0\right)} + \widetilde{\left(e_{j},0\right)}\widetilde{\left(e_{i},0\right)}}{2}f + \widetilde{ \left( 0, \omega \left( e_{i}, e_{j} \right)\right)}f \right)^{2} \notag
\\
& =\sum_{i, j=1}^{\infty} \left(\frac{\widetilde{\left(e_{i},0\right)}\widetilde{\left(e_{j},0\right)} + \widetilde{\left(e_{j},0\right)}\widetilde{\left(e_{i},0\right)}}{2}f \right)^{2}
+
\sum_{i, j=1}^{\infty} \left( \widetilde{ \left( 0, \omega \left( e_{i}, e_{j} \right)\right)}f \right)^{2} \notag
\\
&
+\sum_{i, j=1}^{\infty} \left(\widetilde{\left(e_{i},0\right)}\widetilde{\left(e_{j},0\right)}f + \widetilde{\left(e_{j},0\right)}\widetilde{\left(e_{i},0\right)}f \right)\left( \widetilde{ \left( 0, \omega \left( e_{i}, e_{j} \right)\right)}f \right)
\\
&
=\sum_{i, j=1}^{\infty} \left(\frac{\widetilde{\left(e_{i},0\right)}\widetilde{\left(e_{j},0\right)} + \widetilde{\left(e_{j},0\right)}\widetilde{\left(e_{i},0\right)}}{2}f \right)^{2}
+
\sum_{i, j=1}^{\infty} \left( \widetilde{ \left( 0, \omega \left( e_{i}, e_{j} \right)\right)}f \right)^{2}\notag
\\
& =\| \nabla_{H}^{2}f \|^{2}+
\sum_{i, j=1}^{\infty} \left( \widetilde{ \left( 0, \omega \left( e_{i}, e_{j} \right)\right)}f \right)^{2}. \notag
\end{align*}
where
\[
\nabla_{H}^{2}f:=\sum_{i, j=1}^{\infty} \frac{\widetilde{\left(e_{i},0\right)}\widetilde{\left(e_{j},0\right)}f + \widetilde{\left(e_{j},0\right)}\widetilde{\left(e_{i},0\right)}f }{2}
\]
denotes the symmetrized Hessian. Thus by Lemma \ref{l.4.1}
\begin{equation}\label{e.4.1}
 \| \nabla_{H}^{2}f \|^{2}+\rho_{2}\Gamma^{Z}\left( f \right)\le\sum_{i, j=1}^{\infty} \left(\widetilde{\left(e_{i},0\right)}\widetilde{\left(e_{j},0\right)}f \right)^{2}\le \| \nabla_{H}^{2}f \|^{2}+\| \omega \|_{2}^{2}\Gamma^{Z}\left( f \right).
\end{equation}

Now we want to compute $\Gamma_{2}\left( f \right)$. The first term is simply
\[
\frac{1}{2}L\Gamma\left( f \right)=\sum_{i, j=1}^{\infty} \left( \widetilde{\left(e_{i},0\right)}\widetilde{\left(e_{i},0\right)} \widetilde{\left(e_{j},0\right)} f \right) \widetilde{\left(e_{j},0\right)}f+\left( \widetilde{\left(e_{i},0\right)} \widetilde{\left(e_{j},0\right)} f \right)^{2}.
\]
The second term may be expanded by applying \eqref{e.3.5} twice as follows
\begin{align*}
&\Gamma\left( f, Lf \right)=\sum_{i, j=1}^{\infty} \left(
\widetilde{\left(e_{j},0\right)}\widetilde{\left(e_{i},0\right)}
\widetilde{\left(e_{i},0\right)} f \right) \widetilde{\left(e_{j},0\right)}f
\\
& =\sum_{i, j=1}^{\infty} \left( \widetilde{\left(e_{i},0\right)}\widetilde{\left(e_{i},0\right)} \widetilde{\left(e_{j},0\right)} f \right) \widetilde{\left(e_{j},0\right)}f-2\left(\widetilde{\left(e_{i},0\right)}\widetilde{ \left( 0, \omega \left( e_{i}, e_{j} \right)\right)}f\right)\widetilde{\left(e_{j},0\right)}f.
\end{align*}
Thus, by the upper bound in \eqref{e.4.1} we have that
\begin{align}
\Gamma_{2}\left( f \right)
	&= \frac{1}{2}\left(L\Gamma\left( f \right)-2\Gamma\left( f, Lf \right)\right)
		\notag\\
	&= \sum_{i, j=1}^{\infty} \left( \widetilde{\left(e_{i},0\right)} \widetilde{\left(e_{j},0\right)} f \right)^{2}+2\sum_{i, j=1}^{\infty}\left(\widetilde{\left(e_{i},0\right)}\widetilde{ \left( 0, \omega \left( e_{i}, e_{j} \right)\right)}f\right)\widetilde{\left(e_{j},0\right)}f
\notag\\
	&\label{e.a1}
	\le \| \nabla_{H}^{2}f \|^{2}+\| \omega \|_{2}^{2}\Gamma^{Z}\left(
f \right) + 2\sum_{i,
j=1}^{\infty}\left(\widetilde{\left(e_{i},0\right)}\widetilde{ \left( 0,
\omega \left( e_{i}, e_{j}
\right)\right)}f\right)\widetilde{\left(e_{j},0\right)}f.
\end{align}
Similarly, the lower bound in \eqref{e.4.1} implies that
\begin{equation}
\label{e.a2}
\Gamma_{2}\left( f \right)
	\ge \| \nabla_{H}^{2}f \|^{2}+\rho_{2}\Gamma^{Z}\left( f \right) +
2\sum_{i, j=1}^{\infty}\left(\widetilde{\left(e_{i},0\right)}\widetilde{
\left( 0, \omega \left( e_{i}, e_{j}
\right)\right)}f\right)\widetilde{\left(e_{j},0\right)}f.
\end{equation}
The Cauchy-Schwarz inequality now implies that for any $\nu>0$
\begin{align}
2&\notag \left\vert\sum_{i, j=1}^{\infty}\left(\widetilde{\left(e_{i},0\right)}\widetilde{ \left( 0, \omega \left( e_{i}, e_{j} \right)\right)}f\right)\widetilde{\left(e_{j},0\right)}f \right\vert
\\
	&\notag \le \nu \sum_{j=1}^{\infty}\left(\sum_{i=1}^{\infty}\widetilde{\left(e_{i},0\right)}\widetilde{ \left( 0, \omega \left( e_{i}, e_{j} \right)\right)}f\right)^{2} +\frac{1}{\nu}\sum_{j=1}^{\infty}\left(\widetilde{\left(e_{j},0\right)}f\right)^{2}
\\
	&\label{e.a3}
	\le \nu \| \omega \|_{2}^{2} \Gamma^{Z}_{2}\left( f
		\right)+\frac{1}{\nu}\Gamma\left( f \right),
\end{align}
where the last inequality follows from Lemma \ref{l.gz2}.
Combining \eqref{e.a1}, \eqref{e.a2}, and \eqref{e.a3}
then gives
\begin{align}
\Gamma_{2}\left( f \right)+\nu \| \omega \|_{2}^{2}  \Gamma_{2}^{Z}\left( f \right)
	&\label{e.ly}
	\ge \| \nabla_{H}^{2}f \|^{2}+\rho_{2}\Gamma^{Z}\left( f \right)
		-\frac{1}{\nu}\Gamma\left( f \right) \\
	&\notag
	\ge \rho_{2}\Gamma^{Z}\left( f \right) -\frac{1}{\nu}\Gamma\left( f
		\right).
\end{align}
Finally, taking $\nu$ to be $\frac{\nu}{\| \omega \|_{2}^{2}}$ yields \eqref{curvature}.
\end{proof}

\begin{rem}
In the finite-dimensional case, one could use the Cauchy-Schwarz inequality in
\eqref{e.ly} to give a
lower bound on $\|\nabla_H^2 f\|^{2}$ by $(Lf)^{2}$ with a coefficient
depending on $\mathrm{dim}(H)$.
Such an estimate standardly leads, for example, to Li-Yau type Harnack
inequalities and bounds on logarithmic derivatives of the heat kernel.
See for example \cite{BakBau,BaudoinGarofalo2011,LiYau86}.
\end{rem}

We now prove that the desired commutation formula holds trivially on $G$ and
$G_P$.

\begin{lem}
\label{l.commutation_gamma}
For any smooth cylinder function $f$ on $G$,
\[
\Gamma( f, \Gamma^Z(f))=\Gamma^Z( f , \Gamma(f)).
\]
Similarly, for any $f\in C^\infty(G_P)$,
\[
\Gamma_P( f, \Gamma_P^Z(f))=\Gamma_P^Z( f , \Gamma_P(f)).
\]
\end{lem}

\begin{proof}
The proof is a straightforward computation.
\begin{align*}
\Gamma( f, \Gamma^Z(f))& =\sum_{j=1}^{\infty} \left(\widetilde{\left(e_{j},0\right)  }f \right) \left(\widetilde{\left(e_{j},0\right)  }\Gamma^Z(f)\right) \\
 	& =\sum_{j=1}^{\infty} \sum_{\ell=1}^d
\left(\widetilde{\left(e_{j},0\right)  }f \right)
\left(\widetilde{\left(e_{j},0\right)  }\left(\widetilde{\left( 0,
f_{\ell}\right)} f  \right)^2\right) \\
 	&=2\sum_{j=1}^{\infty} \sum_{\ell=1}^d  \left(\widetilde{\left(e_{j},0\right)
}f \right) \left(\widetilde{\left( 0, f_{\ell}\right)} f  \right)
\left(\widetilde{\left(e_{j},0\right)  }\widetilde{\left( 0, f_{\ell}\right)} f \right) \\
 	&=2\sum_{j=1}^{\infty} \sum_{\ell=1}^d  \left(\widetilde{\left( 0, f_{\ell}\right) }f \right)
\left(\widetilde{\left(e_{j},0\right)
}f \right) \left(\widetilde{\left(f_{\ell},0\right)  }\left(\widetilde{\left( 0,
e_{j}\right) }f \right)\right) \\
 	&=\Gamma^Z( f , \Gamma(f)),
\end{align*}
where we have again used the commutativity of $\widetilde{(e_j,0)}$ and
$\widetilde{( 0, f_l)}$ in the penultimate equality.  The computation for the
second equality is completely analogous.
\end{proof}

\subsection{Functional inequalities on $G_P$}
\label{s.Pin}

Again note that for any $P\in\mathrm{Proj}(W)$, $G_P=PH\times\mathbf{C}$
is a finite-dimensional step 2 stratified Lie group.  If $\{e_j\}_{j=1}^n$
is a orthonormal basis of $PH$, then $\{\widetilde{(e_j,0)}\}_{j=1}^n$
is a H\"ormander set of vector fields on $G_P$.  Thus we may apply
the results of Section \ref{s.finite} to $G_P$.  In particular, by
Proposition \ref{p.rpoin}, the curvature bound for $G_P$ found in
Proposition \ref{p.curvature} implies the following reverse Poincar\'e
inequality holds on all $G_P$.  Here, we let $L_P=\sum_{j=1}^n
\widetilde{(e_j,0)}^2$ and $\{P_t^P\}_{t>0}$ denote the associated
semi-group.  Also, define the function classes
$\mathcal{C}_P$ and $\mathcal{C}_P^+$ analogously for $G_P$ as was
done in Notation \ref{n.C}.

\begin{prop}[Reverse Poincar\'e inequality]
\label{p.Prp}
For any $P\in\mathrm{Proj}(W)$,  $T > 0$, and $f\in \mathcal{C}_P$,
\[
  \Gamma_P(P_T^P f) + \rho_{2,P} T  \Gamma_P^Z(P^P_T f)
	\le \frac{1+\frac{2\| \omega \|_{2,P}^{2}}{\rho_{2,P}} }{T}
		(P^P_T( f^2) -(P^P_Tf)^2).
\]
In particular, the following Reverse Poincar\'e inequality holds
\[
  \Gamma_P(P_T^P f)
	\le \frac{1+\frac{2\| \omega \|_{2,P}^{2}}{\rho_{2,P}} }{T}
		(P^P_T( f^2) -(P^P_Tf)^2).
\]
\end{prop}

Similarly, Theorem \ref{P:linearizedBL} implies that Proposition
\ref{p.curvature} coupled with the commutation relation of Lemma
\ref{l.commutation_gamma} give the following reverse
log Sobolev inequality on all $G_P$.

\begin{thm}[Reverse log Sobolev inequality]
\label{reverse_log_sob}
For any $P\in\mathrm{Proj}(W)$, $T> 0$, and $f\in \mathcal{C}_P^+$,
\[
\Gamma_P( \ln P^P_T f) + \rho_2 T \Gamma_P^Z (\ln P^P_T f)
	\le \frac{ 1+\frac{2 \| \omega \|^2_2}{\rho_2} }{T}
		\left(\frac{P^P_T (f \ln f)}{P^P_Tf} - \ln P^P_T f \right)
 \]
In particular, the following reverse log Sobolev inequality holds
\[
\Gamma_P( \ln P^P_T f)
	\le  \frac{ 1+\frac{2 \| \omega \|^2_2}{\rho_2} }{T}
		\left(\frac{P^P_T (f \ln f)}{ P^P_T f} - \ln P^P_T f \right).
 \]
\end{thm}

The combination of Proposition \ref{p.har} with the reverse log
Sobolev inequality found in Theorem \ref{reverse_log_sob} implies
that the following Harnack type inequalities hold on each $G_P$.

\begin{prop}[Wang type Harnack inequality]
\label{p.Phar}
Let $P\in\mathrm{Proj}(W)$.  Then, for all $T>0$, $x,y \in G_P$, $f\in L^\infty(G_P)$ with $f\ge0$,
and $p\in(1,\infty)$,
\[
( P_T^P f )^p (x) \le P_T^P f^p (y) \exp \left(\left(1 +
	\frac{2\|\omega\|_{2,P}^2}{\rho_{2,P}} \right) \frac{d_P^2(x,y) }{4(p-1)T } \right).
\]
\end{prop}

\section{Heat kernel measure on $G$ and a quasi-invariance theorem}\label{S5}

In this section, we show how the Wang type Harnack inequalities on $G_P$ obtained in the
previous section lead to the quasi-invariance of the subelliptic heat kernel
measure on $G$.  First, we must of course define the heat kernel measure on $G$,
which we define as the end point distribution of a Brownian motion.

\subsection{Brownian motion on $G$}
\label{s.BM}
We define a ``subelliptic'' Brownian motion $\{g_t\}_{t\ge0}$
on $G$ and collect various of its properties.
The primary references for this section are Sections 4 of
\cite{DriverGordina2008} and \cite{DriverGordina2010} and Section 2.5 of
\cite{GordinaMelcher2011}.  Any statements made here without proof are proved
in these references.

Let $\{B_t\}_{t\ge0}$ be a Brownian motion on $W$ with variance determined by
\[
\mathbb{E}\left[\langle B_s,h\rangle_H \langle B_t,k\rangle_H\right]
    = \langle h,k \rangle_H \min(s,t),
\]
for all $s,t\ge0$ and $h,k\in H_*$.  The following is Proposition 4.1 of
\cite{DriverGordina2008} and this result implicitly relies on the fact that
Lemma \ref{l.2} implies that the bilinear form $\omega$ is a Hilbert-Schmidt.
\begin{prop}
\label{p.Mn}
For $P\in\mathrm{Proj}(W)$, let $M_t^P$ denote the continuous $L^2$-martingale on
$\mathbf{C}$ defined by
\[ M_t^P = \int_0^t\omega(PB_s,dPB_s). \]
In particular, if $\{P_n\}_{n=1}^\infty\subset\mathrm{Proj}(W)$
is an increasing sequence of projections and $M_t^n:=M_t^{P_n}$, then there exists an $L^2$-martingale
$\{M_t\}_{t\ge0}$ in $\mathbf{C}$ such that, for all $p\in[1,\infty)$
and $t>0$,
\[ \lim_{n\rightarrow\infty} \mathbb{E}\left[\sup_{\tau\le t}
    \|M_\tau^n-M_\tau\|_\mathbf{C}^p\right]=0, \]
and $M_t$ is independent of the sequence of projections.
\end{prop}

As $M_t$ is independent of the defining sequence of projections, we will
denote the limiting process by
\[
M_t = \int_0^t \omega(B_s,dB_s).
\]

\begin{df}
\label{d.bm}
The continuous $G$-valued process given by
\[g_t = \left( B_t, \frac{1}{2}M_t\right)
	= \left( B_t, \frac{1}{2}\int_0^t \omega(B_s,dB_s)\right).
\]
is a {\em Brownian motion} on $G$.
For $t>0$, let $\nu_t=\mathrm{Law}(g_t)$ denote the {\em heat kernel measure
at time $t$} on $G$.
\end{df}

We include the following proposition (see \cite[Proposition
2.30]{GordinaMelcher2011}) which states that, as the name suggests, the
Cameron-Martin subgroup is a subspace of heat kernel measure 0.
\begin{prop}
For all $t>0$, $\nu_t(G_{CM})=0$.
\end{prop}

Proposition \ref{p.Mn} along with the fact that, for all $p\in[1,\infty)$ and
$t>0$,
\[ \lim_{n\rightarrow\infty} \mathbb{E}\left[\sup_{\tau\le
	t}\|B_\tau-P_nB_\tau\|_W^p \right] = 0 \]
(see for example Proposition 4.6 of \cite{DriverGordina2008})
makes the following proposition clear.

\begin{prop}
\label{p.bmapprox}
For $P\in\mathrm{Proj}(W)$, let $g_t^P$ be the continuous process on $G_P$
defined by
\[ g_t^P = \left(PB_t, \frac{1}{2}\int_0^t\omega(PB_s,dPB_s)\right). \]
Then $g_t^P$ is a Brownian motion on $G_P$ .  In particular, let
$\{P_n\}_{n=1}^\infty\subset\mathrm{Proj}(W)$ be increasing
projections and $g_t^n:=g_t^{P_n}$.  Then, for all $p\in[1,\infty)$
and $t>0$,
\[ \lim_{n\rightarrow\infty} \mathbb{E}\left[\sup_{\tau\le t}
    \|g_\tau^n-g_\tau\|_\mathfrak{g}^p\right]=0. \]
\end{prop}

\begin{nota}
For all $P\in\operatorname*{Proj}\left(  W\right)$ and $t>0$, let $\nu_t^P :=
\mathrm{Law}(g_t^P)$, and for all $n\in\mathbb{N}$ let $\nu_t^n := \mathrm{Law}(g_t^n)
=\mathrm{Law}(g_t^{P_n})$.
\end{nota}

\begin{prop}
\label{p.L}
Let $L$ be as defined in Proposition \ref{p.3.28}.  Then we will call $L$ the
subelliptic Laplacian, and $\frac{1}{2}L$ is the generator for
$\{g_t\}_{t\ge0}$, so that, for any smooth cylinder function
$f:G\rightarrow\mathbb{R}$,
\[ f(g_t) - \frac{1}{2}\int_0^t Lf(g_s)\,ds \]
is a local martingale.
\end{prop}

\begin{cor}
\label{c.L}
Let $f=F\circ\pi_P$ be a cylinder function on $G$ such that $F\in C^2(G_P)$
and there exist $K>0$ and $p<\infty$ such that
\[ |F(h,c)| + \|F'(h,c)\| + \|F''(h,c)\| \le K(1 + \|h\|_{PH} +
	\|c\|_\mathbf{C})^p, \]
for all $(h,c)\in G_P$.  Then
\[ \mathbb{E}[f(g_t)] = f(e) + \frac{1}{2}\int_0^t \mathbb{E}[(Lf)(g_s)]\,ds.
\]
That is,
\[ \nu_t(f) := \int_0^t f\,d\nu_s = f(e) + \frac{1}{2} \int_0^t \nu_s(Lf)\,ds
\]
is a weak solution to the heat equation
\[ \partial_t\nu_t = \frac{1}{2}L\nu_t, \quad \text{ with }
	\lim_{t\downarrow0}\nu_t=\delta_e. \]
\end{cor}

For all projections satisfying H\"ormander's condition, the Brownian
motions on $G_P$ are subelliptic diffusions and thus
their laws are absolutely continuous with respect to the
finite-dimensional reference measure and their transition kernels
are smooth.  The following is Lemma 2.27 of \cite{GordinaMelcher2011}.

\begin{lem}
\label{l.4.8} For all $P\in\operatorname*{Proj}(W)$ and $t>0$, we have
$\nu_{t}^{P}(dx) = p_t^P(x)dx$,
where $dx$ is the Riemannian volume measure (equal to Haar measure) and
$p_{t}^{P}(x)$ is the heat kernel on $G_P.$
\end{lem}

\subsection{Quasi-invariance and Radon-Nikodym derivative estimates}
\label{s.qi}

For now, let us fix  $P\in\mathrm{Proj}(W)$, and recall that by
Proposition \ref{p.Phar}, for all $T>0$, $x,y \in G_P$, $f\ge0$,
and $p\in(1,\infty)$,
\[ ( P_T^P f )^p (x)
	\le P_T^P f^p (y) \exp \left(\left(1 +
	\frac{2\|\omega\|_{2,P}^2}{\rho_{2,P}} \right) \frac{d_P^2(x,y)
	}{4(p-1)T } \right).
\]
Then Lemma \ref{l.harqi2} implies that this estimate is equivalent to
\[
\left(\int_G \left[\frac{p^P_T(y,z)}{p^P_T(x,z)}\right]^{1/(p-1)}
		p^P_T(x,z)\,dz \right)^{p-1}
	\le \exp \left(\left(1 +
		\frac{2\|\omega\|_{2,P}^2}{\rho_{2,P}} \right) \frac{d_P^2(x,y)
		}{4(p-1)T } \right) \\
\]
where $p_T^P$ is the heat kernel on $G_P$.  In particular, for $p\in(1,2)$ and
$q=1/(p-1)\in (1,\infty)$, this implies that
\begin{equation}
\label{e.rnest}
 \left(\int_G \left[\frac{p^P_T(y,z)}{p^P_T(x,z)}\right]^q
		p^P_T(x,z)\,dz \right)^{1/q}
	\le \exp \left(\left(1 +
		\frac{2\|\omega\|_{2,P}^2}{\rho_{2,P}} \right) \frac{qd_P^2(x,y)
		}{4T } \right) \\
\end{equation}
Using the properties of heat kernels on finite-dimensional groups given
in Lemma \ref{l.pt}, we have that
\begin{align*}
\int_G \left[\frac{p^P_T(y,z)}{p^P_T(x,z)}\right]^q
		p^P_T(x,z)\,dz 	
	&= \int_G \left[\frac{p^P_T(yz^{-1})}{p^P_T(xz^{-1})}\right]^q
		p^P_T(xz^{-1})\,dz \\
	&= \int_G \left[\frac{p^P_T(zy^{-1})}{p^P_T(zx^{-1})}\right]^q
		p^P_T(zx^{-1})\,dz.
\end{align*}
Then, for $x=e$, we may rewrite inequality (\ref{e.rnest}) as
\begin{equation}
\label{e.rP}
\left(\int_G \left[\frac{p^P_T(zy^{-1})}{p^P_T(z)}\right]^q
		p^P_T(z)\,dz \right)^{1/q}
	\le \exp \left(\left(1 +
		\frac{2\|\omega\|_{2,P}^2}{\rho_{2,P}} \right) \frac{qd_P^2(e,y)
		}{4T } \right).
\end{equation}

Now, for $y\in G_P$, again we let $R_y:G_P\rightarrow G_P$ denote right
translation.  Then $\nu^P_T\circ R_y^{-1}$ is the
push forward of $\nu_T$ under $R_y$.  For fixed $T>0$, let $J_P^r$ denote the Radon-Nikodym
derivative of $\nu^P_T\circ R_y^{-1}$ with respect to $\nu_T^P$.
Then (\ref{e.rP}) is equivalent to
\begin{equation}
\label{e.right}
\|J_P^r\|_{L^q(G_P,\nu^P_T)} \le
	\exp \left(\left(1 +
		\frac{2\|\omega\|_{2,P}^2}{\rho_{2,P}} \right) \frac{qd_P^2(e,y)
		}{4T} \right).
\end{equation}

Alternatively, again using the properties of $p_T^P$ described in Lemma
\ref{l.pt} and the translation invariance of Haar measure, we could write
\begin{align*}
\int_G \left[\frac{p^P_T(y,z)}{p^P_T(x,z)}\right]^q
	p^P_T(x,z)\,dz 	
	&= \int_G \left[\frac{p^P_T(y^{-1}z)}{p^P_T(x^{-1}z)}\right]^q
		p^P_T(x^{-1}z)\,dz \\
	&= \int_G \left[\frac{p^P_T(y^{-1}xz)}{p^P_T(z)}\right]^q
		p^P_T(z)\,dz.
\end{align*}
Then taking $x=e$ and combining this with the inequality (\ref{e.rnest}) gives
\[
\left(  \int_G \left[\frac{p^P_T(y^{-1}z)}{p^P_T(z)}\right]^q
		p^P_T(z)\,dz\right)^{1/q}
	\le \exp \left(\left(1 +
		\frac{2\|\omega\|_{2,P}^2}{\rho_{2,P}} \right) \frac{qd_P^2(e,y)
		}{4T } \right).
\]
which is equivalent to the left translation analogue
\begin{equation}
\label{e.left}
 \|J_P^l\|_{L^q(G,\nu_T)} \le
	\exp \left(\left(1 +
		\frac{2\|\omega\|_{2,P}^2}{\rho_{2,P}} \right) \frac{qd_P^2(e,y)
		}{4T } \right),
\end{equation}
where $L_y:G_P\rightarrow G_P$ is left translation, $\nu_T^P\circ
L_y^{-1}$ is the push forward of $\nu_T^P$ under $L_y$, and $J_P^l$
denotes the Radon-Nikodym derivative of $\nu_T^P\circ L_y^{-1}$
with respect to $\nu_T$.

Such estimates on the finite-dimensional projection
groups $G_P$ may be used to prove a quasi-invariance theorem on the
infinite-dimensional group $G$.  The following proof is analogous
to the proofs of Theorem 7.2 and 7.3 in \cite{DriverGordina2009}.  Similar methods
were also used  for the elliptic setting in loop groups in
\cite{Driver1997c}, in infinite-dimensional Heisenberg-like groups in
\cite{DriverGordina2008}, and in semi-infinite Lie groups in \cite{Melcher09}.

\begin{thm}[Quasi-invariance of $\nu_t$]
\label{t.qi}
For all $y\in G_{CM}$ and $T>0$, $\nu_T$ is quasi-invariant under
left and right translations by $y$.
Moreover, for all $q\in(1,\infty)$,
\begin{equation}
\label{e.RNest}
\left\|\frac{d(\nu_T\circ R_y^{-1})}{d\nu_T}\right\|_{L^q(G,\nu_T)}
	\le \exp\left( \left(1 + \frac{2\|\omega\|_{2}^2}
		{\rho_{2}} \right) \frac{qd^2(e,y)}{4T} \right)
\end{equation}
and
\[ \left\|\frac{d(\nu_T\circ L_y^{-1})}{d\nu_T}\right\|_{L^q(G,\nu_T)}
	\le \exp\left( \left(1 + \frac{2\|\omega\|_{2}^2}
		{\rho_{2}} \right) \frac{qd^2(e,y)}{4T} \right).\]
\end{thm}

\begin{proof}
Fix $T>0$ and $P_0\in\mathrm{Proj}(W)$.  Let $y\in G_0$ and
$\{P_n\}_{n=1}^\infty$ be an increasing sequence of projections such that
$P_0H\subset P_nH$ for all $n$ and $P_n|_H\uparrow I_H$.  Let $J_n^r:=J_{P_n}^r$ denote the
Radon-Nikodym derivative of $\nu^n_T\circ R_y^{-1}$ with respect to $\nu^n_T$.  Then by the
previous discussion and (\ref{e.right}), we have
\begin{align*}
\|J_n^r\|_{L^q(G_n,\nu_T^n)}
	&\le \exp\left( \left(1 + \frac{2\|\omega\|_{2,n}^2}
		{\rho_{2,n}} \right) \frac{qd_n^2(e,y)}{4T} \right),
\end{align*}
where we let $\|\omega\|_{2,n}:=\|\omega\|_{2,P_n}$ and
$\rho_{2,n}=\rho_{2,P_n}$.

Now let $\sigma$ be an arbitrary horizontal path in $G_{CM}$ such that $\sigma(0)=e$ and
$\sigma(1)=y=(w,c)\in G_0\subset G_n\subset G_{CM}$, and recall the projected horizontal paths
$\sigma_n:[0,1]\rightarrow G_n$ introduced in Lemma \ref{l.dn} given by
\[ \sigma_n(t)
	= (A_n(t),a_n(t))
	:= \left(P_nA(t),\frac{1}{2}\int_0^t \omega(P_nA(s),P_n\dot{A}(s))
		\,ds\right). \]
with $A_n(1)=P_nA(1)=P_nw=w$, and
\[ \varepsilon_n
	:= c- a_n(1)
	= c-\frac{1}{2}\int_0^1 \omega(P_nA(s),P_n\dot{A}(s))\,ds\in\mathbf{C}. \]
It was proved in \eqref{e.dn} that
\[ d_n(e,y) \le \ell(\sigma_n) + C\sqrt{\|\varepsilon_n\|_\mathbf{C}},
\]
for a constant $C$ independent of $n$, and thus
\[\|J_n^r\|_{L^q(G_n,\nu_T^n)} \le \exp\left( \left(1 + \frac{2\|\omega\|_{2,n}^2}
		{\rho_{2,n}} \right)
			\frac{q(\ell(\sigma_n)+C\sqrt{\|\varepsilon_n\|_\mathbf{C}})^2}{4T}
\right). \]

By Proposition \ref{p.bmapprox}, we have that for any bounded continuous
$f$ on $G$,
\begin{equation}
\label{e.5.7}
\int_{G} f d \nu_{t}= \lim_{n\rightarrow\infty}\int_{G_{n}} f\circ i_n \, d\nu_{t}^{n},
\end{equation}
where $ i_n:G_n\rightarrow G$ denotes the inclusion map.  Note that
\begin{align*}
&\int_{G_n} |(f\circ i_n)(xy)|\,d\nu_T^n(x)
	= \int_{G_n} J_n^r(x)|(f\circ i_n)(x)|d\nu_T^n(x) \\
	&\quad\le \|f\circ i_n\|_{L^{q'}(G_n,\nu_T^n)} \exp\left(\left(1 + \frac{2\|\omega\|_{2,n}^2}
		{\rho_{2,n}} \right)
		\frac{q(\ell(\sigma_n)+C\sqrt{\|\varepsilon_n\|_\mathbf{C}})^2}{4T}
		\right),
\end{align*}
where $q'$ is the conjugate exponent to $q$.
As in Lemma \ref{l.dn}, we have via dominated convergence that
$\ell(\sigma_n)\rightarrow\ell(\sigma)$ and
$\|\varepsilon_n\|_\mathbf{C}\rightarrow0$ as $n\rightarrow\infty$.
Thus, allowing $n\rightarrow\infty$ in this last inequality yields
\[ \int_G |f(xy)|\,d\nu_T(x)
	\le \|f\|_{L^{q'}(G,\nu_T)} \exp\left(\left(1 + \frac{2\|\omega\|_{2}^2}
		{\rho_{2}} \right) \frac{q\ell(\sigma)^2}{4T} \right) \]
by \eqref{e.5.7}.
Now optimizing this inequality over all horizontal paths $\sigma$ in $G_{CM}$ connecting $e$ and
$y$ gives
\begin{equation}
\label{e.c}
\int_G |f(xy)|\,d\nu_T(x)
	\le \|f\|_{L^{q'}(G,\nu_T)} \exp\left(\left(1 + \frac{2\|\omega\|_{2}^2}
		{\rho_{2}} \right) \frac{qd(e,y)^2}{4T } \right).
\end{equation}
Thus, we have proved that \eqref{e.c} holds for $f\in BC(G)$ and $y\in
\cup_{P\in\mathrm{Proj}(W)} G_P$.  As this union is dense in $G$ by
Proposition \ref{p.length},
dominated convergence along with the continuity of $d(e,y)$ in $y$ implies
that \eqref{e.c} holds for all $y\in G_{CM}$.

Since the bounded continuous functions are dense in $L^{q'}(G,\nu_T)$ (see for
example Theorem A.1 of \cite{Janson1997}), the inequality in (\ref{e.c}) implies that the
linear functional $\varphi_y:BC(G)\rightarrow\mathbb{R}$ defined by
\[ \varphi_y(f) = \int_G f(xy)\,d\nu_T(x) \]
has a unique extension to an element, still denoted by
$\varphi_y$, of $L^{q'}(G,\nu_T)^*$ which satisfies the bound
\[ |\varphi_y(f)| \le \|f\|_{L^{q'}(G,\nu_T)}
	\exp\left(\left(1 + \frac{2\|\omega\|_{2}^2}
		{\rho_{2}} \right) \frac{qd(e,y)^2}{4T} \right) \]
for all $f\in L^{q'}(G,\nu_T)$.  Since $L^{q'}(G,\nu_T)^*\cong L^q(G,\nu_T)$, there
then exists a function $J_y^r\in
L^q(G,\nu_T)$ such that
\begin{equation}
\label{e.d}
\varphi_y(f) = \int_G f(x)J_y^r(x)\,d\nu_T(x),
\end{equation}
for all $f\in L^{q'}(G,\nu_T)$, and
\[ \|J_y^r\|_{L^q(G,\nu_T)}
	\le \exp\left(\left(1 + \frac{2\|\omega\|_{2}^2}
		{\rho_{2}} \right) \frac{qd(e,y)^2}{4T} \right). \]

Now restricting (\ref{e.d}) to $f\in BC(G)$, we may rewrite this equation as
\begin{equation}
\label{e.last}
\int_G f(x)\,d\nu_T(xy^{-1})
	= \int_G f(x) J_y^r(x)\,d\nu_T(x).
\end{equation}
Then a monotone class argument (again use Theorem A.1 of
\cite{Janson1997}) shows that (\ref{e.last}) is valid for all
bounded measurable functions $f$ on $G$.  Thus,
$d(\nu_T\circ R_y^{-1})/d\nu_T$ exists and is given by $J_y^r$, which is in
$L^q$ for all $q\in(1,\infty)$ and satisfies the bound
(\ref{e.RNest}).

A parallel argument employing the estimate in (\ref{e.left}) gives the
analogous result for $d(\nu_T\circ L_y^{-1})/d\nu_T$.  Alternatively, one
could use the right translation invariance just proved along with the facts that $\nu_T$
inherits invariance under the inversion map $y\mapsto y^{-1}$ from its finite-dimensional projections and that $d(e,y^{-1})=d(e,y)$.
\end{proof}

We may now observe that, by the proof of Lemma \ref{l.harqi2} or by the same
limiting arguments as above, we have the following
Wang type Harnack inequality on $G$.
\begin{cor}[Wang type Harnack inequality on $G$]
\label{c.Har}
For all $T>0$, $x, y \in G_{CM}$, $f\in L^\infty(G,\nu_T)$ with $f\ge0$,
and $p\in \left(1, \infty\right)$
\[( P_T f )^p (x) \le  P_T f^p (y) \exp \left(\left(1 +
	\frac{2\|\omega\|_{2}^2}{\rho_{2}} \right) \frac{d^2(x,y) }{4(p-1)T } \right).
\]
\end{cor}

Our final result is the following corollary which concerns two points. First, we are interested
in the smoothing properties of the semi-group $P_{T}$ in the absence
of a reference measure. One way to approach this is to use a method
similar to the proof of \cite[Theorem 1.1]{Wang2007a}.  A significant
difference is that in \cite{Wang2007a} and some related work, the
proof that the semi-group is strong Feller uses not only
Harnack type inequalities, but also a version of Girsanov's
theorem for the solution of the stochastic differential equation
they consider. Note that we do not presently have path space
quasi-invariance available in our degenerate case. The second point
concerns the fact that the semi-group has smoothing properties, but
only on the Cameron-Martin subgroup which has the heat kernel measure
$0$ as has been shown in \cite[Proposition 2.30]{GordinaMelcher2011}.
This is an infinite-dimensional phenomenon. In the flat abstract
Wiener space setting such a phenomenon has been observed in
\cite{AidaKawabi2001}. In the context of holomorphic functions on
complex abstract Wiener space it has been proved in \cite{Sugita1994b,
Sugita1994a} and on complex infinite-dimensional groups in
\cite{Cecil2008, DriverGordina2008, Gordina2000b}.  The importance
of the strong Feller property for probabilistic potential theory
in infinite dimensions has been discussed in \cite{Gross1967}.  Also,
in \cite{HairerMattingly2006} those authors explore the implications
of a weaker version of the strong Feller property in a hypoelliptic
setting.

\begin{cor}[Strong Feller skeleton]
\label{c.5.19}
For all $T>0$ and $f\in L^\infty(G,\nu_T)$ such that $f\ge0$,
$\left( P_Tf\right)\left(y\right) \to \left( P_T f\right)\left(x\right)$
as $d\left( x, y\right) \to 0$ for $x, y\in G_{CM}$.
\end{cor}

\begin{proof}
Fix $T>0$ and $P_0\in\mathrm{Proj}(W)$.  Let $y\in G_0$ and
$\{P_n\}_{n=1}^\infty$ be an increasing sequence of projections such that
$P_0H\subset P_nH$ for all $n$ and $P_n|_H\uparrow I_H$.  Then
\begin{align*}
\vert \left( P_T^{n}f \right)\left( y \right) - \left( P_T^{n}f \right)\left( x \right)\vert
	&\le \int_{G_{n}} f\left( z \right) \vert p_T^{n} \left( x, z \right)- p_T^{n} \left( y, z \right)\vert dz
\\
& \le \Vert f \Vert_{L^{\infty} \left( G_{n}, \nu_T^{n} \right)}\int_{G_{n}}  \vert p_T^{n} \left( x, z \right)- p_T^{n} \left( y, z \right)\vert dz
\\
& \le \Vert f \Vert_{L^{\infty} \left( G, \nu_T \right)}\int_{G_{n}}  \vert p_T^{n} \left( x, z \right)- p_T^{n} \left( y, z \right)\vert dz.
\end{align*}
We also have that
\begin{align*}
\bigg(&\int_{G_{n}}  \vert p_T^{n} \left( x, z \right)- p_T^{n} \left( y, z
	\right)\vert dz\bigg)^{2}
= \left(\int_{G_{n}}  \left\vert\frac{p_T^{n} \left( x, z \right)}{p_T^{n} \left( y, z \right)}- 1\right\vert p_T^{n} \left( y, z \right)dz\right)^{2}
\\
& \le \int_{G_{n}}\left( \frac{p_T^{n} \left( x, z \right)}{p_T^{n}
\left( y, z \right)}- 1 \right)^{2}p_T^{n} \left( y, z \right) dz
= \int_{G_{n}}\left( \frac{p_T^{n} \left( x, z \right)}{p_T^{n} \left( y, z
\right)} \right)^{2}p_T^{n} \left( y, z \right) dz-1 \\
&\le \exp\left(
\left(1+\frac{2\Vert \omega \Vert_{2, n}^{2}}{ \rho_{2, n}}\right)
\frac{d_{n}^{2}\left( x, y\right)}{2T}\right)-1
\end{align*}
by \eqref{e.rnest} with $q=2$.
Now by the same arguments as in the proof of Theorem \ref{t.qi}, we may show that for any $x, y \in G_{CM}$
\begin{align*}
\vert ( P_T&f )\left( x \right) - \left( P_Tf \right)\left( y
	\right)\vert^2 \\
&\le \Vert f \Vert_{L^{\infty} \left( G, \nu_T \right)}^2\left(
\exp\left( \left(1+\frac{2\Vert \omega \Vert_{2}^{2}}{ \rho_{2}}\right)
\frac{d^{2}\left( x, y\right)}{2T}\right)-1\right)
\\
&\le  \Vert f \Vert_{L^{\infty} \left( G, \nu_T \right)}^2 \left(1+\frac{2\Vert
\omega \Vert_{2}^{2}}{ \rho_{2}}\right) \frac{d^{2}\left( x,
y\right)}{2T}\exp\left( \left(1+\frac{2\Vert \omega \Vert_{2}^{2}}{
\rho_{2}}\right) \frac{d^{2}\left( x, y\right)}{2T}\right),
\end{align*}
since $e^{x}-1 \le x e^{x}$ for any $x>0$. It is clear that the right-hand side tends to $0$ when $d\left( x, y\right) \to 0$.
\end{proof}

\bibliographystyle{amsplain}
\providecommand{\bysame}{\leavevmode\hbox to3em{\hrulefill}\thinspace}
\providecommand{\MR}{\relax\ifhmode\unskip\space\fi MR }
\providecommand{\MRhref}[2]{%
  \href{http://www.ams.org/mathscinet-getitem?mr=#1}{#2}
}
\providecommand{\href}[2]{#2}

\end{document}